\documentclass[11pt,reqno,letter]{amsproc}

\usepackage{amsmath, amsthm, amssymb,color}

\title[Analyticity for the second-grade fluids]{Analyticity and Gevrey-class regularity for the second-grade fluid equations}
\author{Marius Paicu}
\author{Vlad Vicol}

\address{Université Paris-Sud, Laboratoire de Mathématiques, 91405 Orsay Cedex, France }
\email{marius.paicu@math.u-psud.fr}

\address{Department of Mathematics, University of Southern California, Los Angeles, CA 90089, USA}
\email{vicol@usc.edu}
\theoremstyle{plain}
\newtheorem{theorem}{Theorem}[section]

\newtheorem{lemma}[theorem]{Lemma}

\theoremstyle{definition}

\newcommand\llabel[1]{\label{#1}}% {\qquad \textcolor{red}{\textbf{\scriptsize #1}}}\qquad }
\renewcommand{\hat}{\widehat}
\def\T3{{\mathbb T}^3}
\def\Td{{\mathbb T}^d}
\def\Zd{{\mathbb Z}^d}
\def\Fw{\hat{\omega}}
\def\Lms{\Lambda_{m}^{1/s}}
\def\LLs{\Lambda^{1/s} }
\def\sgn{\mathop{\rm sgn} \nolimits}
\def\curl{\mathop{\rm curl} \nolimits}
\def\ddiv{\mathop{\rm div} \nolimits}
\def\rot{\mathop{\rm curl} \nolimits}
\def\LL{\Lambda}
\def\Fu{\hat{u}}
\def\aalpha{\alpha^2}
\def\sss{\zeta}
\begin{document}

%%%%%%%%%%%%%%%%%%%%%%%%%% THE ABSTRACT %%%%%%%%%%%%%%%%%%%%%%%%%%%%%%%%%%%

\begin{abstract}
We address the global persistence of analyticity and Gevrey-class regularity of solutions to the two and three-dimensional visco-elastic second-grade fluid equations. We obtain an explicit novel lower bound on the radius of analyticity of the solutions to the second-grade fluid equations that does not vanish as $t\rightarrow \infty$. Applications to the damped Euler equations are given.
\end{abstract}

%%%%%%%%%%%%%%%%%%%%%%%% Classification and Keywords %%%%%%%%%%%%%%%%%%%%



\maketitle
\tableofcontents
%%%%%%%%%%%%%%%%%%%%%%%%%% The Main Part %%%%%%%%%%%%%%%%%%%%%%%%%%%%%%%%%%

\section{Introduction}\label{sec:intro}\setcounter{equation}{0}
In this paper we address the regularity of an asymptotically smooth system
arising in non-Newtonian fluid mechanics,
which is not smoothing in finite time, but admits a compact global
attractor (in the two-dimensional case). More precisely, we consider the system of visco-elastic second-grade fluids
\begin{align}
  &\partial_t ( u - \aalpha \Delta u) - \nu \Delta u + \curl(u-\aalpha \Delta u) \times u + \nabla p = 0,\llabel{eq:u1}\\
  & \ddiv u =0, \llabel{eq:u2}\\
  & u(0,x) = u_0(x), \llabel{eq:u3}
\end{align}
 where $\alpha>0$ is a material parameter, $\nu \geq 0$ is the kinematic viscosity, the vector field $u$ represents the velocity of the fluid, and the scalar field $p$ represents the pressure. Here $(x,t)\in \Td\times[0,\infty)$, where $\Td = [0, 2\pi]^d$ is the $d$-dimensional torus, and $d\in \{2,3\}$. Without loss of generality we consider velocities that have zero-mean on $\Td$.

Fluids of second-grade are a particular class of non-Newtonian
Rivlin-Ericksen fluids of differential type and the above precise
form has been justified by Dunn and Fosdick \cite{DF}.
The local existence in time, and the uniqueness of strong solutions of the equations \eqref{eq:u1}--\eqref{eq:u3} in a two or three-dimensional bounded domain with no slip boundary conditions has been addressed by Cioranescu and Ouazar \cite{CiO}. Moreover, in the two-dimensional case, they obtained the global in time existence of solutions (see also \cite{CiG,GaSe,GGS,If02}). Moise, Rosa, and Wang \cite{MoRoWa98} have shown later that in two dimensions these equations admit a compact global attractor $\mathcal A_{\alpha}$ (see also \cite{GT87,PRR}). The question of regularity and finite-dimensional behavior of $\mathcal A_{\alpha}$ was studied by Paicu, Raugel, and Rekalo in \cite{PRR}, where it was shown
that the compact global attractor in $H^3({\mathbb T}^2)$ is contained in any Sobolev space $H^m({\mathbb T}^2)$ provided that the material coefficient $\alpha$ is small enough, and the forcing term is regular. Moreover, on the global attractor, the second-grade fluid system can be reduced to a finite-dimensional system of ordinary differential equations with an infinite delay. As a consequence, the existence of a finite number of determining modes for the equation of fluids of grade two was established in \cite{PRR}.

Note that the equations \eqref{eq:u1}--\eqref{eq:u3} essentially differ from the $\alpha$-Navier-Stokes system (cf.~ Foias, Holm, and Titi \cite{Titi1,FHT2}, and references
therein). Indeed, the $\alpha$-Navier-Stokes model (cf.~\cite{FHT2}) contains the very regularizing term $-\nu\Delta (u-\aalpha\Delta u)$, instead of $-\nu\Delta u$, and thus is a
semi-linear problem. This is not the case for the second-grade fluid equations
where the dissipative term is very weak --- it behaves like a damping term --- and the system is not smoothing in finite time. The $\alpha$-models are used, in particular, as an alternative to the usual Navier-Stokes for numerical modeling of turbulence phenomena in pipes and channels. Note that the physics underlying the second-grade fluid equations and the $\alpha$-models
are quite different. There are numerous papers devoted to the asymptotic behavior of the $\alpha$-models, including Camassa-Holm equations, $\alpha$-Navier-Stokes equations, $\alpha$-Bardina equations (cf.~\cite{Titi3,Titi1,FHT2,LaTiti,LiTiti}).

In this paper we characterize the domain of analyticity and Gevrey-class regularity of solutions to the second-grade fluids equation, and of the Euler equation with damping term. We emphasize that the radius of analyticity gives an estimate on the minimal scale in the flow \cite{HKR,Ku}, and it also gives the explicit rate of exponential decay of its Fourier coefficients \cite{FT}. We recall also that the system of second-grade fluids has a unique strong solution $u\in
L^{\infty}_{loc}([0,\infty);H^3)$ in two-dimensional setting (cf.~\cite{CiO}). Thus, opposite to the Navier-Stokes equations, the system of second-grade fluids cannot be smoothing in finite time.

We prove that if the initial data $u_0$ is of Gevrey-class $s$, with $s \geq 1$, then the unique smooth solution $u(t)$ remains of Gevrey-class $s$ for all $t<T_*$, where $T_*\in(0,\infty]$ is the maximal time of existence in the Sobolev norm of the solution. Moreover, for all $\nu \geq 0$ we obtain an explicit lower bound for the real-analyticity radius of the solution, that depends algebraically $\int_{0}^{t} \Vert \nabla u(s) \Vert_{L^\infty} ds$. A similar lower bound on the analyticity radius for solutions to the incompressible Euler equations was obtained by Kukavica and Vicol \cite{KV,KV1} (see also \cite{AM,BBe,BBeZ,Be,LO}). The proof is based on the method of Gevrey-class regularity introduced by Foias and Temam \cite{FT} to study the analyticity of the Navier-Stokes equations (see also \cite{Chemin2,FTi,KV,LaTiti,Lem,LO,OTi98,OTi}). We emphasize that the technique of analytic estimates may be used to obtain the existence of global solutions for the Navier-Stokes equation with some type of large initial data (\cite{CGP,PZ}).

Note that if $\nu>0$, and $d=2$, or if $d=3$ and $u_0$ is small in a certain norm, then $T_*=\infty$, both for the second-grade fluids \eqref{eq:u1}--\eqref{eq:u3}, and for the damped Euler equations \eqref{eq:E1}--\eqref{eq:E3}. The novelty of our result is that in this case the lower bound on the radius of analyticity does not vanish as $t\rightarrow \infty$. Instead, it is bounded from below for all time by a positive quantity that depends solely on $\nu,\alpha$, the analytic norm, and the radius of analyticity of the initial data. In contrast, we note that the shear flow example of Bardos and Titi \cite{BTi} (cf.~\cite{DM}) may be used to construct explicit solutions to the incompressible two and three-dimensional Euler equations (in the absence of damping) whose radius of analyticity is decaying for all time, and hence vanishes as $t\rightarrow \infty$.

The main results of our paper are given bellow (for the definitions see the following sections).
\begin{theorem}\emph{(The three-dimensional case)}
  Fix $\nu ,\alpha > 0$, and assume that $\omega_0$ is of Gevrey-class $s$, for some $s\geq 1$. Then the unique solution $\omega(t) \in C( [0,T^*);L^2(\T3))$ to \eqref{eq:w1}--\eqref{eq:w3} is of Gevrey-class $s$ for all $t<T^*$, where $T^*\in(0,\infty]$ is the maximal time of existence of the Sobolev solution. Moreover, the radius $\tau(t)$ of Gevrey-class $s$ regularity of the solution is bounded from below as
  \begin{align*}
    \tau(t) \geq \frac{\tau_0}{C_0} e^{- C \int_{0}^{t} \Vert \nabla u(s) \Vert_{L^\infty} ds},
  \end{align*}
  where $C>0$ is a dimensional constant, and $C_0>0$ has additional explicit dependence on the initial data, $\alpha$, and $\nu$ via \eqref{eq:C0def*} below.
\end{theorem}
In the two dimensional case we obtain the global in time control of the radius of analyticity, which is moreover uniform in $\alpha$. This allows us to prove the convergence as $\alpha \to 0$ of the solutions of the second-grade fluid to solutions of the corresponding Navier-Stokes equations in analytic norms (cf.~Section~\ref{sec:NSEconv}). The convergence of solutions to the Euler-$\alpha$ equations to the corresponding Euler equations, in the limit $\alpha \to 0$, has been addressed in \cite{LiTiti}.
\begin{theorem}\emph{(The two-dimensional case)}
Fix $\nu> 0$, $0\leq \alpha < 1$, and assume that $u_0$ is of Gevrey-class $s$ for some $s\geq 1$, with radius $\tau_0 > 0$. Then there exists a unique global in time Gevrey-class $s$ solution $u(t)$ to \eqref{eq:u1}--\eqref{eq:u3}, such that for all $t\geq 0$ the radius of Gevrey-class regularity is bounded from below by
   \begin{align*}
    \tau(t) \geq \frac{\tau_0}{1+C_0\tau_0},
   \end{align*}where $C_0>0$ is a constant depending on $\nu$ and the initial data via \eqref{eq:C0def} below.
\end{theorem}

%Organization of the paper...

%%%%%%%%%%%%%%%%%%%%%%%%%%%%%%%%%%%%%%%%
\section{Preliminaries}
In this section we introduce the notations that are used throughout the paper. We denote the usual Lebesgue spaces by $L^p(\Td)=L^p$, for $1\leq p \leq \infty$. The $L^2$-inner product is denoted by $\langle \cdot,\cdot\rangle$. The Sobolev spaces $H^r(\Td) = H^r$ of {\it mean-free functions} are classically characterized in terms of the Fourier series
\begin{align*}
  H^r(\Td) =\{ v(x) = \sum\limits_{k\in\Zd} \hat{v}_k e^{i k\cdot x}\, :\, \overline{\hat{v}_k} = \hat{v}_{-k},\ \hat{v}_0 = 0,\ \Vert v \Vert_{H^r}^2 = (2\pi)^3 \sum\limits_{k\in\Zd} (1+|k|^2)^r |\hat{v}_k|^2 <\infty\}.
\end{align*}
We let $\lambda_1>0$ be the first positive eigenvalue of the Stokes operator, which in the periodic setting coincides with $-\Delta$ \cite{CF,T}. For simplicity we consider $\Td = [0,2\pi]^d$, and hence $\lambda_1 = 1$. The Poincar\'e inequality then reads $\Vert v \Vert_{L^2}\leq \Vert \nabla v \Vert_{L^2}$ for all $v\in H^1$. Throughout the paper we shall denote by $\Lambda$ the operator $(-\Delta)^{1/2}$, i.e., the Fourier multiplier operator with symbol $|k|$. We will denote by $C$ a generic sufficiently large positive dimensional constant, which does not depend on $\alpha,\nu$. Moreover, the curl of a vector field $v$ will be denoted by $\curl v =  \nabla \times v$.

\subsection{Dyadic decompositions and para-differential calculus}

Fix a smooth nonnegative radial function $\chi$ with support in the ball  $\{\vert\xi\vert\leq\frac{4}{3}\},$
which is identically $1$ in $\{\vert\xi\vert\leq\frac{3}{4}\},$
and such that the map $r\mapsto \chi(|r|)$ is non-increasing over ${\mathbb R}_+.$
Let  $\varphi(\xi)=\chi(\xi/2)-\chi(\xi).$
We classically have
\begin{equation}\label{eq:dyadique}
\sum_{q\in {\mathbb Z}}\varphi(2^{-q}\xi)=1\quad\hbox{for all}\quad
\xi\in {\mathbb R}^d\setminus\{0\}.
\end{equation}
We define the spectral localization operators
$\Delta_q$ and  $ S_q$ ($q\in {\mathbb Z}$) by
$$
\Delta_q\;u:=\varphi(2^{-q}D)u=\sum\limits_{k\in \Zd} \hat u(k)e^{ikx}\varphi(2^{-q}|k|)$$
and
 $$S_q\,u:=\chi(2^{-q}D)u=\sum\limits_{k\in \Zd}\hat u(k) e^{ikx}\chi(2^{-q}|k|).
$$
We have the following quasi-orthogonality
property~:
\begin{equation}\label{presortho}
\Delta_k\Delta_q u\equiv 0 \quad\mbox{if}\quad\vert k-q\vert\geq 2
\quad\mbox{and}\quad\Delta_k( S_{q-1}u\Delta_qv)\equiv
0\quad\mbox{if} \quad\vert k-q\vert\geq 5.
\end{equation}
We recall the very useful \emph{Bernstein inequality}.
\begin{lemma}  Let $n\in {\mathbb N},$ $1\leq p_1\leq p_2\leq\infty$ and
$\psi\in C_c^\infty({\mathbb R}^d).$ There exists a constant $C$
depending only on $n,d$ and ${\rm Supp}\,\psi$ such that
$$\hskip1cm\|D^n\psi(2^{-q}D)u\|_{L^{p_2}}
\leq C2^{q\bigl(n+N\bigl(\frac 1{p_{1}}-\frac
1{p_{2}}\bigr)\bigr)}\|\psi(2^{-q}D)u\|_{L^{p_1}},$$
and
$$C^{-1}2^{q\bigl(n+N\bigl(\frac 1{p_{1}}-\frac
1{p_{2}}\bigr)\bigr)}\|\varphi(2^{-q}D)u\|_{L^{p_1}}\leq\sup_{|\alpha|=n}\|\partial^\alpha\varphi(2^{-q}D)u\|_{L^{p_2}}
\leq C2^{q\bigl(n+N\bigl(\frac 1{p_{1}}-\frac
1{p_{2}}\bigr)\bigr)}\|\varphi(2^{-q}D)u\|_{L^{p_1}}.$$
\end{lemma}

In order to obtain optimal bounds on the nonlinear terms in a system, we use the paradifferential calculus,
a tool which was introduced by J.-M. Bony
in \cite{B}.
More precisely, the product of two functions
$f$ and $g$ may be decomposed according to
\begin{equation}\label{eq:decompobony}
fg= T_fg+ T_gf+ R(f,g)
\end{equation}
where the paraproduct operator $ T$ is defined by
the formula
$$\displaystyle{\, T_fg:=\sum_q S_{q\!-\!1}f\,\Delta_q g},$$
and the remainder operator, $ R,$ by
 $$\displaystyle{\, R(f,g)
:=\sum_q \Delta_qf\tilde\Delta_qg}
\quad\text{with}\quad\tilde\Delta_q:=\Delta_{q\!-\!1} +\Delta_q
+\Delta_{q\!+\!1}.$$

%%%%%%%%%%%%%%%%%%%%%%%%%%%%%%%%%%%%%%%%%%%%%%%%%
\subsection{Analytic and Gevrey-class norms} \label{sec:analytic:def}
Classically, a $C^\infty(\Td)$ function $v$ is in the {\it Gevrey-class} $s$, for some $s>0$ if there exist $M,\tau>0$ such that
\begin{align*}
  |\partial^\beta v (x)|\leq M \frac{\beta!^s}{\tau^{|\beta|}},
\end{align*}for all $x \in \Td$, and all multi-indices $\beta \in {\mathbb N}_{0}^{3}$. We will refer to $\tau$ as the {\it radius of Gevrey-class regularity} of the function $v$. When $s=1$ we recover the class of real-analytic functions, and the {\it radius of analyticity} $\tau$ is (up to a dimensional constant) the radius of convergence of the Taylor series at each point. When $s>1$ the Gevrey-classes consist of $C^\infty$ functions which however are not analytic. It is however more convenient in PDEs to use an equivalent characterization, introduced by Foias and Temam \cite{FT} to address the analyticity of solutions of the Navier-Stokes equations. Namely, for all $s\geq 1$ the Gevrey-class $s$ is given by
\begin{align*}
  \bigcup_{\tau>0} {\mathcal D}( \Lambda^r e^{\tau \Lambda^{1/s}} )
\end{align*}for any $r\geq 0$, where
\begin{align}
  \Vert \Lambda^r e^{\tau \Lambda^{1/s}} v \Vert_{L^2}^2 = (2\pi)^3 \sum\limits_{k\in\Zd} |k|^{2r} e^{2\tau |k|^{1/s}} |\hat{v}_k|^2.
\end{align}
See \cite{CF,FTi,FT,Ku,KV,KV1,LO,OTi,T} and references therein for more details on Gevrey-classes.

%For $\tau>0$ and $s\in(0,1]$ we define the norms, and corresponding spaces of Gevrey-class functions by
%\begin{align}
%  \Vert \omega \Vert_{X_{s,\tau}} = \Vert e^{\tau \Lambda^s} \omega \Vert_{L^2},
%\end{align}
%and
%\begin{align}
%  \Vert \omega \Vert_{Y_{s,\tau}} = \Vert\Lambda^{s/2} e^{\tau \Lambda^s} \omega \Vert_{L^2}.
%\end{align}The fact that these norms are equivalent to the norms defined in Section~\ref{sec:large3d} is a consequence of the triangle inequality, and so if $\omega \in X_{s,\tau}$, then $\omega$ is of Gevrey-class $s$, with radius $\tau$ (cf.~\cite{FT,KV,LO}).
%

\subsection{Vorticity Formulation}\label{sec:vorticity}
It is convenient to consider the evolution of the vorticity $\omega$, which is defined as \begin{align}\llabel{eq:w:def}
\omega = \curl(u - \aalpha \Delta u) =  (I - \aalpha \Delta) \curl u.
 \end{align}It follows from \eqref{eq:u1}--\eqref{eq:u2}, that $\omega$ satisfies the initial value problem
\begin{align}
  &\partial_t \omega - \nu \Delta (I - \aalpha \Delta)^{-1} \omega + (u \cdot \nabla) \omega = (\omega \cdot \nabla) u, \llabel{eq:w1}\\
  & \ddiv \omega = 0, \llabel{eq:w2}\\
  & \omega(0,x) = \omega_0(x) = \curl (u_0 - \aalpha \Delta u_0) \llabel{eq:w3}
\end{align}on $\Td\times(0,\infty)$. Additionally, if $d=2$, $\omega$ is a scalar, and the right side of \eqref{eq:w1} is zero. Denote by ${\mathcal R}_\alpha$ the operator
\begin{align}
{\mathcal R}_\alpha = (-\Delta) (I -\aalpha \Delta)^{-1}.
\end{align}It follows from Plancherel's theorem, that for all $v\in L^2$ we have
\begin{align}
  \frac{1}{1+\alpha} \Vert v \Vert_{L^2} \leq \Vert {\mathcal R}_{\alpha} v \Vert_{L^2} \leq \frac{1}{\alpha} \Vert v \Vert_{L^2}. \label{eq:Restimate}
\end{align}The velocity is obtained from the vorticity by solving  the elliptic problem
\begin{align}
  \ddiv u =0,\ \curl u = (I - \aalpha \Delta)^{-1} \omega, \int_{\T3} u =0, \llabel{eq:u:def1}
\end{align}which in turn classically gives that
\begin{align}
  u = K\ast (I - \aalpha \Delta)^{-1} \omega = {\mathcal K}_\alpha \omega, \llabel{eq:u:def2}
\end{align} where $K$ is the periodic Biot-Savart kernel. Combined with \eqref{eq:Restimate}, the above implies that
\begin{align}
\Vert u \Vert_{H^3}\leq \frac{C}{\alpha} \Vert \omega \Vert_{L^2},\label{eq:Kestimate}
\end{align}for some universal constant $C>0$. Note that when $\alpha\rightarrow 0$ the above estimate becomes obsolete.

%%%%%%%%%%%%%%%%%%%%%%%%%%%%%%%%%%%%%%%%

\section{The two-dimensional case}

\subsection{The case $\alpha$ large}\label{sec:2d:small}\setcounter{equation}{0}
In the two-dimensional case, the evolution equation \eqref{eq:w1} for $\omega$ does not include the term $\omega \cdot \nabla u$, which makes the problem tangible, in analogy to the two-dimensional Euler equations. The main result below gives the global well-posedness of solutions evolving from Gevrey-class data, whose radius $\tau(t)$ does not vanish as $t\rightarrow \infty$.
\begin{theorem}\label{thm:2d:big}
   Fix $\nu, \alpha> 0$, and assume that $\omega_0 \in {\mathcal D}(e^{\tau_0 \Lambda^{1/s}})$, for some $s\geq 1$, and $\tau_0 > 0$. Then there exists a unique global in time Gevrey-class $s$ solution $\omega(t)$ to \eqref{eq:w1}--\eqref{eq:w3}, such that for all $t\geq 0$ we have $\omega(t) \in {\mathcal D}(e^{\tau(t) \Lambda^{1/s}})$, and moreover we have the lower bound
   \begin{align}\llabel{eq:thm:2d:big}
    \tau(t) \geq \tau_0 e^{ -C M_0 \int_{0}^{t}  e^{-\nu s/(2+2\aalpha)} ds / \alpha} \geq \tau_0 e^{-C (2+2\aalpha) M_0 / (\alpha \nu) },
   \end{align}where $M_0 = \Vert e^{\tau_0\Lambda^{1/s}}\omega_0 \Vert_{L^2}$, and $C$ is a universal constant.
\end{theorem}
%\begin{remark}
%  We note that the lower bound \eqref{eq:thm:2d:small} is independent of $t$, and is strictly positive for all $\nu>0$.
%\end{remark}
\begin{proof}[Proof of Theorem~\ref{thm:2d:big}]
  We take the $L^2$-inner product of $\partial_t \omega + \nu {\mathcal R}_{\alpha} \omega+ (u\cdot \nabla) \omega = 0$ with $e^{2\tau \LLs}$ and obtain
  \begin{align}
    \frac{1}{2} \frac{d}{dt} \Vert e^{\tau \LLs} \omega \Vert_{L^2}^2 - \dot{\tau} \Vert \LL^{1/2s} e^{\tau \LLs} \omega \Vert_{L^2}^2 + \langle e^{\tau \LLs} {\mathcal R}_{\alpha} \omega, e^{\tau \LLs} \omega \rangle = - \langle e^{\tau \LLs}( u \cdot \nabla \omega), e^{\tau \LLs} \omega \rangle.
  \end{align}Note that the Fourier multiplier symbol of the operator ${\mathcal R}_\alpha$ is an increasing function of $|k|\geq 1$, and therefore by Plancherel's theorem and Parseval's identity we have
\begin{align*}
  \langle e^{\tau \LLs} {\mathcal R}_\alpha \omega, \Lambda e^{\tau \LLs} \omega\rangle  &= (2\pi)^2 \sum\limits_{k\in {\mathbb Z}^2 \setminus \{ 0\}} \frac{|k|^2}{1+\aalpha |k|^2} |\Fw_k|^2 e^{2\tau |k|^{1/s}}\\
  &\geq \frac{(2\pi)^2}{1+\aalpha}  \sum\limits_{k\in {\mathbb Z}^2 \setminus \{ 0\}} |\Fw_k|^2 e^{2\tau |k|^{1/s}} = \frac{1}{1+\aalpha}\Vert e^{\tau \LLs} \omega \Vert_{L^2}^2 .
\end{align*}We therefore have the {\it a priori} estimate
  \begin{align}
    \frac 12 \frac{d}{dt} \Vert e^{\tau \LLs} \omega \Vert_{L^2}^2 - \dot{\tau} \Vert \LL^{1/2s} e^{\tau \LLs} \omega \Vert_{L^2}^2 + \frac{\nu}{1+\aalpha} \Vert e^{\tau \LLs} \omega \Vert_{L^2}^2 \leq |\langle u\cdot \nabla \omega, e^{2\tau \LLs} \omega\rangle |.\llabel{eq:2d:ODE1}
  \end{align}The following lemma gives a bound on the convection term on the right of \eqref{eq:2d:ODE1} above.
  \begin{lemma}\label{lemma:2d:convection}
    There exists a dimensional constant $C>0$ such that for all $\omega \in {\mathcal D}(\LL^{1/2s} e^{\tau \LLs})$, and divergence free $u={\mathcal K}_{\alpha} \omega$, we have
\begin{align}
\left| \langle u\cdot \nabla \omega, e^{2\tau \LLs} \omega \rangle \right| \leq \frac{C\tau}{\alpha} \Vert e^{\tau \LLs} \omega \Vert_{L^2} \Vert \LL^{1/2s} e^{\tau \LLs} \omega \Vert_{L^2}^2.\llabel{eq:2d:convection}
\end{align}
  \end{lemma} Therefore, by \eqref{eq:2d:ODE1} and \eqref{eq:2d:convection}, if we chose $\tau$ that satisfies
  \begin{align}\llabel{eq:2d:taucondition}
  \dot{\tau} + \frac{C \tau}{\alpha} \Vert e^{\tau \LLs} \omega \Vert_{L^2} = 0,
  \end{align} then we have
  \begin{align*}
  \frac 12 \frac{d}{dt} \Vert \omega \Vert_{X_{s,\tau}}^2  + \frac{\nu}{1+\aalpha} \Vert \omega \Vert_{X_{s,\tau}}^2 \leq 0,
  \end{align*} and hence
  \begin{align}
    \Vert e^{\tau(t) \LLs} \omega(t) \Vert_{L^2} \leq \Vert e^{\tau_0 \LLs} \omega_0 \Vert_{L^2} e^{-\gamma t},
  \end{align}where we have denoted $\gamma = \nu/(2 + 2 \aalpha)$. The above estimate and condition \eqref{eq:2d:taucondition} show that
  \begin{align}
    \tau(t) \geq \tau_0 e^{-\frac{C}{\alpha} \Vert e^{\tau_0 \LLs} \omega_0 \Vert_{L^2}  \int_{0}^{t}e^{-\gamma s}ds} \geq \tau_0 e^{-C (2+2\aalpha) \Vert e^{\tau_0 \LLs} \omega_0 \Vert_{L^2}/ (\nu \alpha)},
  \end{align}which concludes the proof of the theorem. The above {\it a priori} estimates are made rigorous using a classical Fourier-Galerkin approximating sequence. We omit further details.
\end{proof}

%%%%%%%%%%%%%%%%%%%%%%%%%%%%%%%%%%%%%%%%%%%%%%%%%%%

\subsection{The case $\alpha$ small} The lower bound \eqref{eq:thm:2d:big} on the radius of Gevrey-class regularity goes to $0$ as $\alpha\rightarrow 0$. In this section we give a new estimate on $\tau(t)$, in the case when $\alpha$ is small.
\begin{theorem}\label{thm:2d:small}
Fix $\nu> 0$, $0\leq \alpha < 1$, and assume that $\curl u_0 \in {\mathcal D}(\Delta e^{\tau_0 \Lambda^{1/s}})$, for some $s\geq 1$, and $\tau_0 > 0$. Then there exists a unique global in time Gevrey-class $s$ solution $u(t)$ to \eqref{eq:u1}--\eqref{eq:u3}, such that for all $t\geq 0$ we have $u(t) \in {\mathcal D}(e^{\tau(t) \Lambda^{1/s}})$, and moreover we have the lower bound
   \begin{align}\llabel{eq:thm:2d:small}
    \tau(t) \geq \frac{\tau_0}{1+C_0\tau_0},
   \end{align}where $C_0 = C_0(\nu, \Vert  u_0 \Vert_{H^3}, \Vert \LL e^{\tau_0 \LLs} \curl u_0 \Vert_{L^2}, \Vert e^{\tau \LLs} \curl \Delta u_0 \Vert_{L^2}$) is given explicitly in \eqref{eq:C0def}.
\end{theorem}
\begin{proof}[Proof of Theorem~\ref{thm:2d:small}]For simplicity of the presentation, we give the proof in the case $s=1$.
Taking the $L^2$-inner product of \eqref{eq:u1} with $- e^{2\tau \LL} \rot \Delta u$, we obtain
\begin{align}
& \frac 12 \frac{d}{dt} \left( \Vert \LL e^{\tau \LL} \rot u \Vert_{L^2}^2 + \aalpha \Vert e^{\tau \LL} \rot \Delta u \Vert_{L^2}^2 \right) + \nu \Vert e^{\tau \LL} \rot \Delta u \Vert_{L^2}^2 \notag \\
&\qquad \qquad \qquad  - \dot{\tau} \left( \Vert \LL^{3/2} e^{\tau \LL} \rot u \Vert_{L^2}^2 + \aalpha \Vert \LL^{1/2} e^{\tau \LL} \rot \Delta u \Vert_{L^2}^2\right) \leq T_1 + T_2,\llabel{eq:ode1}
\end{align}where
\begin{align}
  T_1 = \aalpha \left| \langle e^{\tau \LL} \left( (u\cdot \nabla) \Delta \rot u\right), e^{\tau \LL} \Delta \rot u \rangle\right|,
\end{align}and
\begin{align}
  T_2 = \left| \langle \LL e^{\tau \LL} \left( (u\cdot \nabla) \rot u \right), \LL e^{\tau \LL} \rot u \rangle \right|.
\end{align}
The upper bounds for $T_1$ and $T_2$ are given in the following lemma.
\begin{lemma}\label{lemma:2d:new:estimate}
Let $\nu,\tau>0$, $0\leq \alpha<1$, and $u$ be such that $\curl u \in {\mathcal D}(\LL^{5/2} e^{\tau \LL})$. Then
\begin{align}\llabel{eq:2d:T1}
T_1 \leq \frac{\nu}{4} \Vert e^{\tau \LL} \rot \Delta u \Vert_{L^2}^2 + \frac{C \alpha^4 \tau^2}{\nu} \Vert \LL^{1/2} e^{\tau \LL} \rot \Delta u \Vert_{L^2}^2 \Vert e^{\tau \LL} \rot \Delta u \Vert_{L^2}^2,
\end{align}
and
\begin{align}\llabel{eq:2d:T2}
T_2 &\leq \frac{\nu}{4} \Vert e^{\tau \LL} \Delta \rot u \Vert_{L^2}^2 + \frac{C}{\nu^3} \Vert \curl u \Vert_{L^2}^4 \Vert \LL e^{\tau \LL} \rot u \Vert_{L^2}^2\notag\\
 & \qquad \qquad \qquad + \frac{C\tau^2}{\nu} \Vert \LL^{3/2} e^{\tau \LL} \rot u \Vert_{L^2}^2 \Vert \LL e^{\tau \LL} \rot u \Vert_{L^2}^2,
\end{align}
where $C>0$ is a universal constant.
\end{lemma}
We give the proof of the above lemma in the Appendix (cf.~Section~\ref{app:2d:new:estimate}). Assuming that estimates \eqref{eq:2d:T1} and \eqref{eq:2d:T2} are proven, we obtain from \eqref{eq:ode1} that
\begin{align}
& \frac 12 \frac{d}{dt} \left( \Vert \LL e^{\tau \LL} \rot u \Vert_{L^2}^2 + \aalpha \Vert e^{\tau \LL} \rot \Delta u \Vert_{L^2}^2 \right) + \frac{\nu}{2} \Vert e^{\tau \LL} \rot \Delta u \Vert_{L^2}^2 \notag\\
& \qquad \qquad \leq  \frac{C}{\nu^3} \Vert \curl u \Vert_{L^2}^4 \Vert \LL e^{\tau \LL} \rot u \Vert_{L^2}^2  + \left( \dot{\tau} + \frac{C\tau^2}{\nu} \Vert \LL e^{\tau \LL} \rot u \Vert_{L^2}^2\right) \Vert \LL^{3/2} e^{\tau \LL} \rot u \Vert_{L^2}^2 \notag \\
& \qquad \qquad + \aalpha \left( \dot{\tau} + \aalpha\frac{C\tau^2}{\nu} \Vert e^{\tau \LL} \rot \Delta u \Vert_{L^2}^2\right) \Vert \LL^{1/2} e^{\tau \LL} \rot \Delta u \Vert_{L^2}^2.\llabel{eq:ode2}
\end{align}Define
\begin{align*}
  Z(t) = \Vert \LL e^{\tau \LL} \rot u \Vert_{L^2}^2 \\
 \end{align*}
  and
 \begin{align*}
  W(t)= \Vert e^{\tau \LL} \rot \Delta u \Vert_{L^2}^2.
\end{align*}We let $\tau$ be decreasing fast enough so that
\begin{align}
  \dot{\tau}(t) + \frac{C \tau(t)^2}{\nu} W(t) = 0, \llabel{eq:tau1}
\end{align}which by the Poincar\'e inequality implies
\begin{align*}
\dot{\tau} + \frac{C\tau^2}{\nu} \Vert \LL e^{\tau \LL} \rot u \Vert_{L^2}^2\leq 0,
\end{align*}
and also
\begin{align*}
\dot{\tau} + \aalpha\frac{C\tau^2}{\nu} \Vert e^{\tau \LL} \rot \Delta u \Vert_{L^2}^2\leq 0,
\end{align*}
since by assumption $\alpha \leq 1$. It follows that for all $0 \leq \alpha \leq 1$ we have
\begin{align}
  \frac 12 \frac{d}{dt} (Z+\aalpha W) + \frac{\nu}{2} W &\leq \frac{C}{\nu^3} \Vert \curl u \Vert_{L^2}^4 Z \llabel{eq:ode31}\\
  &\leq \frac{C}{\nu^3}\Vert \curl u\Vert_{L^2}^4(Z+\aalpha W). \llabel{eq:ode32}
\end{align}We recall that $\omega = \curl (I - \aalpha \Delta) u$ solves the equation
\begin{align}
\partial_t \omega +\nu {\mathcal R}_\alpha \omega + (u\cdot \nabla)\omega=0
\end{align}
which by the classical energy estimates implies
\begin{align}
\frac 12\frac{d}{dt}\|\omega(t)\|_{L^2}^2+\frac{\nu}{1+\aalpha}\|\omega(t)\|^2_{L^2}\leq 0
\end{align}and therefore
\begin{align}
\|\omega(t)\|_{L^2}^2\leq \|\omega_0\|_{L^2}^2 e^{-2\gamma t}
\end{align}
where $\gamma=\nu/(2+2\aalpha)$. Using that $0\leq\alpha< 1$ and
\begin{align}
\|\omega\|_{L^2}^2=\|\rot u\|^2_{L^2}+2\aalpha\|\Delta u\|_{L^2}^2+\alpha^4\|\rot\Delta u\|_{L^2}^2
\end{align}  we obtain the exponential decay rate
\begin{align}
\|\curl u(t)\|_{L^2}\leq C \|u_0\|_{H^3} e^{-\gamma t}.\llabel{eq:ode4}
\end{align}Combining \eqref{eq:ode32} and \eqref{eq:ode4}, and using $\alpha \leq 1$, we get
\begin{align}
  Z(t) &\leq (Z(0)+\aalpha W(0)) e^{\frac{C}{\nu^3} \int_0^t\|\curl u(s)\|^4_{L^2}\, ds}\notag\\
  &\leq (Z(0)+\aalpha W(0))e^{\frac{C}{4 \gamma \nu } \Vert u_0 \Vert_{H^3}^4} \leq (Z(0)+ W(0))e^{\frac{C}{\nu^4} M_0^4}, \llabel{eq:ode5}
\end{align}where we have denoted $M_0 = \Vert u_0 \Vert_{H^3}$. Plugging the above bound in \eqref{eq:ode31} and integrating in time, we obtain
\begin{align}
Z(t)+\aalpha W(t)+\frac{\nu}{2}\int_0^t W(s)\, ds &\leq (Z(0) + W(0)) \left( 1 + \frac{C}{\nu^3} e^{CM_0^4 / \nu^4} \int_0^t \|\curl u(s)\|_{L^2}^4\, ds\right)\notag\\
&\leq (Z(0) + W(0)) \left( \frac{1}{\nu^2} + \frac{C M_0^4}{\nu^6} e^{CM_0^4 / \nu^4}\right) \nu^2 =C_0 \nu^2,\llabel{eq:C0def}
\end{align}where $C_0 =C_0(\nu,\Vert u_0 \Vert_{H^3},Z(0),W(0))>0$ is a constant depending on the data. Thus, by the construction of $\tau$ in \eqref{eq:tau1} and the above estimate, by possibly enlarging $C_0$, we have the lower bound
\begin{align}
  \tau(t) &= \left( \frac{1}{\tau_0} + \frac{C}{\nu} \int_0^t W(s)\, ds \right)^{-1}\geq  \frac{\tau_0}{1 + \tau_0 C_0},\llabel{eq:taubound}
\end{align}thereby proving \eqref{eq:thm:2d:small}. We note that this lower bound is independent of $t\geq 0$, and $0\leq \alpha \leq 1$. This concludes the {\it a priori} estimates needed to prove Theorem~\ref{thm:2d:small}. The formal construction of the real-analytic solution is standard and we omit details. The proof of the theorem in the case $s>1$ follows {\it mutatis mutandis}.
\end{proof}

\subsection{Convergence to the Navier-Stokes equations as $\alpha\rightarrow 0$} \label{sec:NSEconv}
In this section we compare in an analytic norm the solutions of the second-grade fluids equations with those of the corresponding Navier-Stokes equations, in the limit as $\alpha$ goes to zero. The fact that the analyticity radius for the solutions of the second-grade fluids is bounded from bellow by a positive constant, for all positive time, will play a fundamental role. We consider $a>0$ and $u_0$ such that $e^{a\Lambda}u_0\in H^3({\mathbb T}^2)$. We recall that the Navier-Stokes equations
\begin{align}
&\partial_t u-\nu\Delta u+\rot u\times u + \nabla p=0 \notag \\
&\ddiv u=0\\
&u|_{t=0}=u_0,\notag
\end{align}
have a unique global regular solution when $u_0\in L^2({\mathbb T}^2)$. Moreover, this solution is analytic for every $t>0$, and if $e^{\delta \Lambda}u_0\in H^3$ one can prove that $e^{\delta\Lambda} u(t)\in H^3$ for all $t>0$ (for example, one can use the same proof as in the one in Section~\ref{sec:damped-euler}). Let $u_\alpha$ denote the solution of the second-grade fluids equations; then $z=u_\alpha-u$ satisfies the following equation
\begin{align}
&\partial_t (z-\alpha^2\Delta z)-\nu\Delta z+\rot z\times u_\alpha+\rot u\times z+\nabla(p_\alpha-p)=\alpha^2\partial_t\Delta u+\alpha^2\rot\Delta u_\alpha\times u_\alpha \notag\\
&\ddiv z=0\\
&z(0)=0.\notag
\end{align}
The following product Sobolev estimate (see \cite{Chemin2}) will prove to be very useful
\begin{equation}
\label{lege-produs}\|e^{\delta\Lambda}(ab)\|_{H^{s_1+s_2-1}({\mathbb T}^2)}\leq \|e^{\delta\Lambda}a\|_{H^{s_1}({\mathbb T}^2)}\|e^{\delta\Lambda}b\|_{H^{s_2}({\mathbb T}^2)},
\end{equation}
 where $s_1+s_2>0$, $s_1<1$, $s_2<1$.
Applying $e^{\delta\Lambda}$ with $0<\delta<a$ fixed but small enough (given for example by \eqref{eq:taubound}) to the equation, denoting by $z^\delta(t)=e^{\delta\Lambda} z(t)$, and considering the $L^2({\mathbb T}^2)$ energy estimates, using \eqref{lege-produs}, the Young inequality, and the classical Sobolev inequalities, we obtain the following estimate
\begin{align*}
&\frac 12\frac{d}{dt}(\|z^\delta\|^2_{L^2}+\alpha^2\|\nabla z^\delta\|_{L^2}^2)+\nu\|\nabla z^\delta\|_{L^2}^2\\
& \qquad \qquad \qquad  \leq \frac{C\alpha^4}{\nu}\|\partial_t\nabla u^\delta\|_{L^2}^2+\frac{C}{\nu}\|u_\alpha^\delta\|_{H^{\frac 12}}^2\|z^\delta\|_{H^{\frac 12}}^2+\frac{\nu}{50}\|\nabla z^\delta\|_{L^2}^2\\
&\qquad \qquad \qquad \qquad +\alpha^2\|\rot\Delta u_\alpha^\delta\|_{L^2}\|u_\alpha^\delta\|_{H^{\frac 12}}\|z^\delta\|_{H^{\frac 12}}+\|\rot u^\delta\|_{L^2}\|z^\delta\|^2_{H^{\frac 12}}\\
&\qquad \qquad \qquad \leq \frac{C\alpha^4}{\nu}\|\partial_t\nabla u^\delta\|_{L^2}^2+\frac{C}{\nu}\|u_\alpha^\delta\|_{L^2}\|\nabla u_\alpha^\delta\|_{L^2}\|z^\delta\|_{L^2}\|\nabla z^\delta\|_{L^2}+\frac{\nu}{50}\|\nabla z^\delta\|_{L^2}^2\\
&\qquad \qquad \qquad \qquad +\alpha^2 \|\rot\Delta u_\alpha^\delta\|_{L^2}\|u_\alpha^\delta\|_{L^2}^{\frac 12}\|\nabla u_\alpha^\delta\|_{L^2}^{\frac 12}\|z^\delta\|_{L^2}^{\frac 12}\|\nabla z^\delta\|_{L^2}^{\frac 12}+\|u^\delta\|_{H^1}\|z^\delta\|_{L^2}\|\nabla z^\delta\|_{L^2}\\
&\qquad \qquad \qquad \leq \frac{C\alpha^4}{\nu}\|\partial_t\nabla u^\delta\|_{L^2}^2+\frac{\nu}{4}\|\nabla z^\delta\|_{L^2}^2
+\frac{C}{\nu^2}\|u_\alpha^\delta\|_{L^2}^2\|\nabla u_\alpha^\delta\|_{L^2}^2\|z^\delta\|_{L^2}^2\\
&\qquad \qquad \qquad  \qquad+\frac{C\alpha^4}{\nu}\|\rot\Delta u_\alpha^\delta\|_{L^2}^2\|u_\alpha^\delta\|_{L^2}\|\nabla u_\alpha^\delta\|_{L^2}+\frac{\nu}{4}\|z^\delta\|_{L^2}^2+\frac{C}{\nu}\|u^\delta\|_{H^1}^2\|z^\delta\|_{L^2}^2.
\end{align*}
From the above estimate and the Poincar\'{e} inequality, we deduce that for $t \geq 0$
\begin{align*}
&\frac{d}{dt} (\|z^\delta\|_{L^{2}}^{2} + \alpha^2
\|\nabla z^\delta\|_{L^{2}}^{2}) + \gamma
\big(\|z^\delta\|_{L^{2}}^{2} + \alpha^2
\|\nabla z^\delta\|_{L^{2}}^{2} \big) \leq \left( \frac{C}{\nu}\|u^\delta\|_{H^1}^2 + \frac{C}{\nu^2}
\|u_{\alpha}^\delta\|_{L^{2}}^2\|\nabla
u_{\alpha}^\delta\|_{L^{2}}^2 \right) \|z^\delta\|_{L^{2}}^2\\
&\qquad \qquad \qquad \qquad + \frac{C\alpha^4}{\nu}\big( \|\partial_t \nabla u^\delta \|_{L^{2}}^{2}+\| \rot \Delta u_{\alpha}^\delta\|_{L^{2}}^2
\|u_{\alpha}^\delta\|_{L^{2}}\|\nabla u_{\alpha}^\delta\|_{L^{2}}\big),
\end{align*} where $0 < \gamma = \nu/(2+2\aalpha)$.
Integrating this inequality from  $0$ to $t$ and using
the Gr\"onwall inequality, we obtain
\begin{align*}
& \|z^\delta(t)\|_{L^{2}}^{2} + \alpha^2
\|\nabla z^\delta(t)\|_{L^{2}}^{2} \leq \int_{0}^{t}\left(\frac{C}{\nu}\|u^\delta\|_{H^1}^2+\frac{C}{\nu^2}
\|u_{\alpha}^\delta\|_{L^{2}}^2\|\nabla
u_{\alpha}^\delta\|_{L^{2}}^2 \right) \|z^\delta\|_{L^{2}}^2 ds\\
&\qquad \qquad \qquad  +\frac{C\alpha^4}{\nu} \int_{0}^{t} \exp \left( \gamma(s-t)\right) \Big(\|\partial_t
\nabla u^\delta
 (s)\|_{L^{2}}^{2}+  \| \rot \Delta
u_{\alpha}^\delta\|_{L^{2}}^2\|u_{\alpha}^\delta\|_{L^{2}}\|\nabla u_{\alpha}^\delta\|_{L^{2}}\Big)ds.
\end{align*} Using one more time the Gr\"onwall
lemma, we deduce from the above estimate that, for $t \geq 0$
\begin{align}
&\|z^\delta(t)\|_{L^{2}}^{2} + \alpha^2
\|\nabla z^\delta(t)\|_{L^{2}}^{2} \leq \exp \left(\int_{0}^{t}\Big( \frac{C}{\nu}\|u^\delta\|_{H^1}^2 + \frac{C}{\nu^2}
\|u_{\alpha}^\delta\|_{L^{2}}^2\|\nabla
u_{\alpha}^\delta\|_{L^{2}}^2 \Big) ds \right) \notag\\
& \qquad \qquad \times  \frac{C\alpha^4}{\nu} \int_{0}^{t} \exp(\gamma(s-t)) \Bigg(
\|\partial_t \nabla u^\delta
 (s)\|_{L^{2}}^{2} + \| \rot \Delta
u_{\alpha}^\delta\|_{L^{2}}^2
\|u_{\alpha}^\delta\|_{L^{2}}\|\nabla u_{\alpha}^\delta\|_{L^{2}} \Bigg)ds.\llabel{estimez2}
\end{align}
 We recall the estimate \eqref{eq:C0def} on $u_\alpha^\delta$, which gives
\begin{equation}
\llabel{apriori4}
\|\Delta u_\alpha^\delta\|_{L^2}^2+\alpha^2\|\rot\Delta u_\alpha^\delta\|_{L^2}^2+\nu \int_0^t\|\rot\Delta u_\alpha^\delta\|_{L^2}^2\leq M_0.
\end{equation}
The equation on $u_\alpha$ gives that $\partial_t u_\alpha=(I-\alpha^2\Delta)^{-1}[\nu\Delta u_\alpha-{\mathbb P}(\curl(u_\alpha-\alpha^2\Delta u_\alpha)\times u_\alpha)]$, and using the estimate \eqref{lege-produs}, the previous bound and the fact that the operator $\alpha\nabla(I-\alpha^2\Delta)^{-1}$ is uniformly bounded on $L^2({\mathbb T}^2)$, we obtain that $\alpha\|\partial_t \nabla u^\delta_\alpha\|_{L^2}\leq C M_0$.
When  $\alpha \leq 1$, the inequality \eqref{estimez2} together with the above uniform bounds
 and the corresponding property for the Navier-Stokes equation, namely $\int_0^t\|u^\delta\|_{H^1}^2\leq M$, implies that
\begin{equation}
\llabel{estimez3}
\|z^\delta(t)\|_{L^{2}}^{2} + \alpha
\|\nabla z^\delta(t)\|_{L^{2}}^{2} \leq   \alpha^2 K_0 e^ K_1,
\end{equation}
where $K_0$ and $K_1$ are positive constants depending only on $\| e^{a\Lambda}u_0\|_{H^{3}}$.
 Thus, we obtain the convergence in the analytic norm as $\alpha \to 0$ of the solution of the second-grade fluid to the solutions of Navier-Stokes equations, with same analytic initial data $u_0$, such that $e^{a \Lambda} u_0\in H^3$.

\section{The three-dimensional case}\label{sec:3d}\setcounter{equation}{0}
\subsection{Global in time results for small initial data}\label{sec:3d:small}
In this section we state our main result in the case $\nu >0$, with small initial data. There exists a global in time solution whose Gevrey-class radius is bounded from below by a positive constant for all time. A similar result for small data is obtained in \cite{Sang}.
\begin{theorem}\label{thm:3d:small}
   Fix $\nu, \alpha> 0$, and assume that $\omega_0 \in {\mathcal D}(\LL^{1/2s} e^{\tau_0 \LLs})$, for some $s\geq 1$, and $\tau_0 > 0$. There exists a positive sufficiently large dimensional constant $\kappa$, such that if
   \begin{align}\llabel{thm:3d:condition}
     \kappa \Vert \omega_0 \Vert_{L^2} \leq \frac{\nu \alpha}{2(1+\aalpha)},
   \end{align} then there exists a unique global in time Gevrey-class $s$ solution $\omega(t)$ to \eqref{eq:w1}--\eqref{eq:w3}, such that for all $t\geq 0$ we have $\omega(t) \in {\mathcal D}(e^{\tau(t) \LLs})$, and moreover we have the lower bound
   \begin{align}\llabel{eq:thm:3d:small}
    \tau(t) \geq \tau_0 e^{-\kappa (4+4\aalpha) M_0 / (\nu \alpha)}
   \end{align}for all $t\geq 0$, where $M_0 = \Vert e^{\tau_0 \LLs} \omega_0 \Vert_{L^2}$.
\end{theorem}
The smallness condition \eqref{thm:3d:condition} ensures that $\Vert \omega(t) \Vert_{L^2}$ decays exponentially in time, and hence by the Sobolev and Poincar\'e inequalities the same decay holds for $\Vert \nabla u(t) \Vert_{L^\infty}$. Therefore, as opposed to the case of large initial data treated in Section~\ref{sec:3d:big}, in this case there is no loss in expressing the radius of Gevrey-class regularity in terms of the vorticity $\omega(t)$. It is thus more transparent to prove Theorem~\ref{thm:3d:small} by just using the operator $\Lambda$ (cf.~\cite{LO}), instead of using the operators $\Lambda_m$ (cf.~\cite{KV}) which are used to prove Theorem~\ref{thm:3d:big} below.

\begin{proof}[Proof of Theorem~\ref{thm:3d:small}]
Similarly to \eqref{eq:2d:ODE1}, we have the {\it a priori} estimate
\begin{align}
\frac 12 \frac{d}{dt} \Vert e^{\tau \LLs} \omega \Vert_{L^2}^2 + \frac{\nu}{1+\aalpha} \Vert e^{\tau \LLs} \omega \Vert_{L^2}^2 &\leq \dot{\tau} \Vert \LL^{1/2s} e^{\tau \LLs} \omega \Vert_{L^2}^2 \notag\\
& + | ( u\cdot \nabla \omega, e^{2\tau \LLs} \omega)| + |(\omega\cdot \nabla u, e^{2\tau \LLs} \omega)|. \llabel{eq::ode1}
\end{align}
The convection term and the vorticity stretching term are estimated in the following lemma.
\begin{lemma}\label{lemma:new}
  There exists a positive dimensional constant $C$ such that for $\omega \in Y_{s,\tau}$, and $u={\mathcal K}_\alpha$ is divergence-free, we have
  \begin{align}
    | ( u\cdot \nabla \omega, e^{2\tau \LLs} \omega)|  \leq \frac{C \tau}{\alpha} \Vert e^{\tau \LLs} \omega \Vert_{L^2} \Vert \LL^{1/2s} e^{\tau \LLs} \omega \Vert_{L^2}^2,\llabel{eq:convection:term}
  \end{align}
  and
  \begin{align}
    |(\omega\cdot \nabla u, e^{2\tau \LLs} \omega)| \leq \frac{C}{\alpha} \Vert \omega \Vert_{L^2} \Vert e^{\tau \LLs} \omega \Vert_{L^2}^2 + \frac{C \tau}{\alpha} \Vert e^{\tau \LLs} \omega \Vert_{L^2} \Vert \LL^{1/2s} e^{\tau \LLs} \omega \Vert_{L^2}^2.\llabel{eq:stretching:term}
  \end{align}
\end{lemma}The proof of the above lemma is similar to \cite[Lemma~8]{LO}, but for the sake of completeness a sketch is given in the Appendix (cf.~Section~\ref{app:new}).

The smallness condition \eqref{thm:3d:condition} implies via the Sobolev and Poincar\'e inequalities that $\Vert \nabla u_0 \Vert_{L^\infty} \leq \nu/(2 + 2 \aalpha)$, if $\kappa$ is chosen sufficiently large. Let $\gamma = \nu/(2+2 \aalpha)$. It follows from standard energy inequalities that $\Vert \omega(t) \Vert_{L^2} \leq \Vert \omega_0 \Vert_{L^2} e^{-\gamma t/2}\leq \Vert \omega_0 \Vert_{L^2}$. Combining this estimate with \eqref{eq::ode1}, \eqref{eq:convection:term}, and \eqref{eq:stretching:term}, we obtain
\begin{align}
\frac 12 \frac{d}{dt} \Vert e^{\tau \LLs} \omega \Vert_{L^2}^2 + \gamma \Vert e^{\tau \LLs} \omega \Vert_{L^2}^2 \leq  \left( \dot{\tau} + \frac{C\tau}{\alpha} \Vert e^{\tau \LLs} \omega \Vert_{L^2} \right)  \Vert \Lambda^{1/2s} e^{\tau \LLs} \omega \Vert_{L^2}^2,
\end{align}where we have used that $\kappa$ was chosen sufficiently large, i.e., $\kappa \geq C$. The above {a-priori} estimate gives the global in time Gevrey-class $s$ solution $\omega(t) \in {\mathcal D}(e^{\tau(t) \LLs})$, if the radius of Gevrey-class regularity $\tau(t)$ is chosen such that
\begin{align}
\dot{\tau} + \frac{C \tau}{\alpha } \Vert e^{\tau \LLs} \omega \Vert_{L^2} \leq 0.
\end{align}
Since under this condition we have
\begin{align*}
  \Vert e^{\tau(t) \LLs} \omega(t) \Vert_{L^2} \leq \Vert e^{\tau_0 \LLs} \omega_0 \Vert_{L^2} e^{-\gamma t/2}
\end{align*}for all $t\geq 0$, it is sufficient to let $\tau(t)$ be such that
$\dot{\tau} + C M_0  e^{-\gamma t/2} \tau /\alpha= 0$, where we let $M_0 = \Vert e^{\tau_0 \LLs} \omega_0 \Vert_{L^2}$. We obtain
\begin{align}
  \tau(t) = \tau_0 e^{-C M_0 \int_{0}^{t}  e^{-\gamma s/2} ds / \alpha},
\end{align}and in particular the radius of analyticity does not vanish as $t\rightarrow \infty$, since it is bounded as
\begin{align}
  \tau(t) \geq \tau_0 e^{-2 C M_0 / (\gamma \alpha)} = \tau_0 e^{-C M_0 (4 + 4\aalpha) / (\nu \alpha)},
\end{align} for all $t\geq 0$, thereby concluding the proof of Theorem~\ref{thm:3d:small}. \end{proof}

\subsection{Large initial data}\label{sec:3d:big}
The main theorem of this section deals with the case of large initial data, where only the local in time existence of solutions is known (cf.~\cite{CiG,CiO}). We prove the persistence of Gevrey-class regularity: as long as the solution exists does not blow-up in the Sobolev norm, it does not blow-up in the Gevrey-class norm. Similarly to the Euler equations, the finite time blow-up remains an open problem.
\begin{theorem}\label{thm:3d:big}
  Fix $\nu ,\alpha > 0$, and assume that $\omega_0$ is of Gevrey-class $s$, for some $s\geq 1$. Then the unique solution $\omega(t) \in C( [0,T^*);L^2(\T3))$ to \eqref{eq:w1}--\eqref{eq:w3} is of Gevrey-class $s$ for all $t<T^*$, where $T^*\in(0,\infty]$ is the maximal time of existence of the Sobolev solution. Moreover, the radius $\tau(t)$ of Gevrey-class $s$ regularity of the solution is bounded from below as
  \begin{align}
  \llabel{eq:thm1}
    \tau(t) \geq \frac{\tau_0}{C_0} e^{- C \int_{0}^{t} \Vert \nabla u(s) \Vert_{L^\infty} ds},
  \end{align}
  where $C>0$ is a dimensional constant, and $C_0>0$ has additional explicit dependence on the initial data, $\alpha$, and $\nu$ via \eqref{eq:C0def*} below.
\end{theorem}
We note that the radius of Gevrey-class regularity is expressed in terms of $ \Vert \nabla u \Vert_{L^\infty}$, as opposed to an exponential in terms of higher Sobolev norms of the velocity. Hence Theorem~\ref{thm:3d:big} may be viewed as a blow-up criterion: if the initial data is of Gevrey-class $s$ (its Fourier coefficients decay at the exponential rate $e^{-\tau_0 |k|^{1/s}}$), and at time $T_*$ the Fourier coefficients of the solution $u(T_*)$ do not decay sufficiently fast, then the solution blows up at $T_*$.

To prove Theorem~\ref{thm:3d:big}, let us first introduce the functional setting. For fixed $s\geq 1$, $\tau \geq 0$, and $m\in \{1,2,3\}$, we define via the Fourier transform the space
\begin{align*}
{\mathcal D}(\Lambda_m e^{\tau \Lms}) = \biggl\{ \omega\in C^\infty(\Td) &: \ddiv \omega = 0, \int_{\Td} \omega = 0,\\ &\left\Vert{\Lambda_m e^{\tau \Lms} \omega}\right\Vert_{L^2}^2 = (2\pi)^d \sum\limits_{k\in {{\mathbb
Z}^d}} |k_m|^2 e^{2\tau |k_m|^{1/s}} |\hat{\omega}_k|^2 <
\infty\biggr\},
\end{align*}
where $\Fw_k$ is the $k^{th}$ Fourier coefficient of $\omega$, and $\Lambda_m$ is the Fourier-multiplier operator with symbol $|k_m|$. For $s,\tau$ as before, also define the normed
spaces $Y_{s,\tau}\subset X_{s,\tau}$ by
\begin{align}
  X_{s,\tau} = \bigcap\limits_{m=1}^{3} {\mathcal D}(\Lambda_m e^{\tau \Lms}),
\qquad \left\Vert \omega \right\Vert_{X_{s,\tau}}^2
=\sum\limits_{m=1}^{3} \left\Vert{\Lambda_m e^{\tau \Lms} \omega}\right\Vert_{L^2}^2, \llabel{eq:Xdef}
\end{align}
and
\begin{align}
  Y_{s,\tau} = \bigcap\limits_{m=1}^{3} {\mathcal D}(\Lambda_{m}^{1+s/2} e^{\tau \Lms}),
\qquad \left\Vert \omega \right\Vert_{Y_{s,\tau}}^2
=\sum\limits_{m=1}^{3} \left\Vert{\Lambda_{m}^{1+s/2} e^{\tau \Lms} \omega}\right\Vert_{L^2}^2,\llabel{eq:Ydef}
\end{align}
It follows from the triangle inequality that if $\omega \in X_{s,\tau}$ then $\omega$ is a function of Gevrey-class $s$, with radius proportional to $\tau$ (up to a dimensional constant). If instead of the $X_{s,\tau}$ norm we use $\Vert \LL e^{\tau \LL^{1/s}} \omega \Vert_{L^2}$ (cf.~\cite{CF,LO}), then the lower bound for the radius of Gevrey-class regularity will decay exponentially in $\Vert \omega \Vert_{H^1}$ (i.e., a higher Sobolev norm of the velocity). It was shown in \cite{KV} that using the spaces $X_{s,\tau}$ it is possible give lower bounds on $\tau$ that depend algebraically on the higher Sobolev norms of $u$, and exponentially on $\Vert \nabla u(t) \Vert_{L^\infty}$, which in turn gives a better estimate on the analyticity radius.

\begin{proof}[Proof of Theorem~\ref{thm:3d:big}]
Assume that the initial datum $\omega_0$ is of Gevrey-class $s$, for some $s\geq 1$, with $\omega_0 \in Y_{s,\tau_0}$, for some $\tau_0 = \tau(0) > 0$. We take the $L^2$-inner product of \eqref{eq:w1} with $\Lambda_{m}^{2} e^{2\tau(t) \Lms} \omega(t)$ and obtain
\begin{align*}
  (\partial_t \omega, \Lambda_{m}^{2} e^{2\tau \Lms} \omega) + \nu ( {\mathcal R}_\alpha \omega, \Lambda_{m}^{2} e^{2\tau \Lms} \omega) + (u \cdot \nabla \omega, \Lambda_{m}^{2} e^{2\tau \Lms} \omega) - (\omega \cdot \nabla u,\Lambda_{m}^{2} e^{2\tau \Lms} \omega).
\end{align*}For convenience we have omitted the time dependence of $\tau$ and $\omega$. The above implies
\begin{align}\llabel{eq:ode1}
 (\partial_t \Lambda_m e^{\tau \Lms} \omega, \Lambda_m e^{\tau \Lms} \omega)  &- \dot{\tau} (\Lambda_{m}^{1+s/2} e^{\tau \Lms} \omega, \Lambda_{m}^{1+s/2} e^{\tau \Lms} \omega) + \nu ( {\mathcal R}_\alpha \Lambda_m e^{\tau \Lms} \omega, \Lambda_m e^{\tau \Lms} \omega) \notag\\
  & \qquad \qquad =   -(u \cdot \nabla \omega, \Lambda_{m}^{2} e^{2\tau \Lms} \omega) + (\omega \cdot \nabla u,\Lambda_{m}^{2} e^{2\tau \Lms} \omega).
\end{align} Note that the Fourier multiplier symbol of the operator ${\mathcal R}_\alpha$ is an increasing function of $|k|\geq 1$, and therefore by Plancherel's thorem and Parseval's identity we have
\begin{align*}
  ( {\mathcal R}_\alpha \Lambda_m e^{\tau \Lms} \omega, \Lambda_m e^{\tau \Lms} \omega) &= (2\pi)^3 \sum\limits_{k\in {\mathbb Z}^3 \setminus \{ 0\}} \frac{|k|^2}{1+\aalpha |k|^2} |k_m|^2 |\Fw_k|^2 e^{2\tau |k|^s}\\
  &\geq \frac{(2\pi)^3}{1+\aalpha}  \sum\limits_{k\in {\mathbb Z}^3 \setminus \{ 0\}} |k_m|^2 |\Fw_k|^2 e^{2\tau |k|^s} = \frac{1}{1+\aalpha}\Vert\Lambda_m e^{\tau \Lms} \omega \Vert_{L^2}^2 .
\end{align*} The above estimate combined with \eqref{eq:ode1} gives for all $m\in \{1,2,3\}$, the {\it a-priori} estimate
\begin{align}\llabel{eq:ode2}
\frac 12 \frac{d}{dt} \Vert \Lambda_m e^{\tau \Lms} \omega \Vert_{L^2}^2 + \frac{\nu}{1+\aalpha} \Vert \Lambda_m e^{\tau \Lms} \omega \Vert_{L^2}^2 - \dot{\tau} \Vert \Lambda_{m}^{1+s/2} e^{\tau \Lms}\omega \Vert_{L^2}^2 \leq T_1 + T_2,
\end{align} where we have denoted
\begin{align}
  T_1 = \left| (u \cdot \nabla \omega, \Lambda_{m}^{2} e^{2\tau \Lms} \omega) \right|,\ \mbox{and}\
  T_2 = \left| (\omega \cdot \nabla u, \Lambda_{m}^{2} e^{2\tau \Lms} \omega) \right|. \llabel{eq:Tdef}
\end{align}
The convection term $T_1$, and the vorticity stretching term $T_2$ are estimated using the fact that $\ddiv u =0$, and that $u = {\mathcal K}_\alpha \omega$.
\begin{lemma}\label{lemma:KV}
For all $m\in \{1,2,3\}$ and $\omega \in Y_{s,\tau}$, we have
\begin{align}
T_1 + T_2 &\leq C \left\Vert{\nabla u}\right\Vert_{L^\infty}
\left\Vert \omega \right\Vert_{X_{s,\tau}}^2 + \frac{C}{\alpha} (1+\tau) \left\Vert
\omega \right\Vert_{H^1}^2 \left\Vert \omega \right\Vert_{X_{s,\tau}}  \notag \\
&+ \left(C \tau \left\Vert{\nabla u}\right\Vert_{L^\infty} + \frac{C \tau^2}{\alpha} \left\Vert \omega \right\Vert_{H^1} + \frac{C \tau^2}{\alpha} \left\Vert \omega \right\Vert_{X_{s,\tau}}\right) \left\Vert \omega \right\Vert_{Y_{s,\tau}}^2, \llabel{eq:mainlemma}
\end{align}
where $C>0$ is a dimensional constant.
\end{lemma}This lemma in the context of the Euler equations was proven by Kukavica and Vicol \cite[Lemma 2.5]{KV}, but for the sake of completeness we sketch the proof in the Appendix (cf. Section~\ref{app:KV}). The novelty of this lemma is that the term $\Vert  \nabla u \Vert_{L^\infty}$ is paired with $\tau$, while the term $\Vert \omega \Vert_{H^1}$ is paired with $\tau^2$. This gives the exponential dependence on the gradient norm and the algebraic dependence of the Sobolev norm. By summing over $m=1,2,3$ in \eqref{eq:ode2}, and using \eqref{eq:mainlemma}, we have proven the {\it a-priori} estimate
\begin{align}\llabel{eq:ode3}
\frac 12 \frac{d}{dt} \Vert \omega \Vert_{X_{s,\tau}}^2 + \frac{\nu}{1+\aalpha} \Vert \omega \Vert_{X_{s,\tau}}^2 &\leq  C \Vert \nabla u \Vert_{L^\infty} \Vert \omega \Vert_{X_{s,\tau}}^2 + \frac{C}{\alpha} (1+\tau) \left\Vert
\omega \right\Vert_{H^1}^2 \Vert \omega \Vert_{X_{s,\tau}} \notag \\
& + \left(\dot{\tau} + C\tau \left\Vert{\nabla u}\right\Vert_{L^\infty} + \frac{C \tau^2}{\alpha} \left\Vert \omega \right\Vert_{H^1} + \frac{C \tau^2}{\alpha} \left\Vert \omega \right\Vert_{X_{s,\tau}}\right) \left\Vert \omega \right\Vert_{Y_{s,\tau}}^2.
\end{align}Therefore, if the radius of Gevrey-class regularity is chosen to decay fast enough so that
\begin{align}
\dot{\tau} + C\tau \left\Vert{\nabla u}\right\Vert_{L^\infty} + \frac{C \tau^2}{\alpha} \left\Vert \omega \right\Vert_{H^1} + \frac{C \tau^2}{\alpha} \left\Vert \omega \right\Vert_{X_{s,\tau}} \leq 0,
\end{align}then for all $\nu > 0$ we have
\begin{align}
  \frac{d}{dt} \Vert \omega \Vert_{X_{s,\tau}}+2 \gamma\Vert\omega\Vert_{X_{s,\tau}}&\leq  C \Vert \nabla u \Vert_{L^\infty} \Vert \omega \Vert_{X_{s,\tau}} + \frac{C}{\alpha} (1+\tau_0) \left\Vert
\omega \right\Vert_{H^1}^2,
\end{align}
where as before $\gamma=\nu/(2+2\aalpha)$. Hence by Gr\"onwall's inequality
\begin{align}
  \Vert \omega(t) \Vert_{X_{s,\tau(t)}} \leq M(t) e^{-2\gamma t} \left( \Vert \omega_0 \Vert_{X_{s,\tau_0}} + \frac{C}{\alpha} (1+\tau_0) \int_{0}^{t} \Vert \omega(s) \Vert_{H^1}^2 e^{2\gamma s} M(s)^{-1} ds \right).\llabel{eq:Gevreynormestimate}
\end{align} where for the sake of compactness we have denoted
\begin{align*}
  M(t) = e^{C \int_{0}^{t} \Vert \nabla u(s) \Vert_{L^\infty}ds}.
\end{align*}
Thus it is sufficient to consider the Gevrey-class radius $\tau(t)$ that solves
\begin{align}
  & \dot{\tau}(t) + C\tau(t) \left\Vert{\nabla u(t)}\right\Vert_{L^\infty}+ \frac{C \tau^2(t)}{\alpha} \left\Vert \omega(t) \right\Vert_{H^1}  \notag\\
   & \qquad + \frac{C \tau^2(t)}{\alpha} M(t) e^{-2\gamma t}\left( \Vert \omega_0 \Vert_{X_{s,\tau_0}} + \frac{C}{\alpha}(1+\tau_0) \int_{0}^{t} \Vert \omega(s) \Vert_{H^1}^2 e^{2\gamma s} M(s)^{-1} ds \right) = 0.\llabel{eq:tau:ode:}
\end{align}The explicit dependence of $\tau$ is hence algebraically on $\Vert \omega \Vert_{H^1}$ and exponentially on $\Vert \nabla u \Vert_{L^\infty}$ via
\begin{align}
  \tau(t) = M(t)^{-1} \Bigg( \frac{1}{\tau_0} &+ \frac{C}{\alpha} \int_{0}^{t} \Vert \omega(s)\Vert_{H^1} M(s)^{-1} + e^{-2\gamma s} \Vert \omega_0 \Vert_{X_{s,\tau_0}}\, ds\notag\\
   & + \frac{C(1+\tau_0)}{\alpha^2} \int_{0}^{t} e^{-2\gamma s} \int_{0}^{s} \Vert \omega(s') \Vert_{H^1}^2 M(s')^{-1} e^{2\gamma s'}\, ds'\, ds\Bigg)^{-1}.\llabel{eq:tau:}
\end{align}
A more compact lower bound for $\tau(t)$ is obtained by noting that if $\nu \geq 0$ we have
\begin{align}
\Vert \omega(t) \Vert_{H^1}^2 \leq M(t) e^{-2\gamma t} \Vert \omega_0 \Vert_{H^1}^2\label{eq:sobolev:growth:bound}
\end{align}for all $t\geq 0$. Assuming \eqref{eq:sobolev:growth:bound} holds, if $\nu >0$ (and hence $\gamma > 0$), then
\begin{align}
  \tau(t) &\geq M(t)^{-1} \left( \frac{1}{\tau_0} + C  \frac{\Vert \omega_0 \Vert_{H^1} + \Vert \omega_0\Vert_{X_{s,\tau_0}}}{\alpha \gamma} + C \frac{(1+\tau_0)\Vert \omega_0 \Vert_{H^1}^2}{4\alpha^2 \gamma^2}\right)^{-1} \geq \frac{\tau_0}{C_0} M(t)^{-1},
\end{align}where the constant $C_0 = C_0 (\nu,\alpha,\tau_0,\omega_0)$ is given explicitly by
\begin{align}
  C_0 = 1 + C \tau_0 (\Vert \omega_0 \Vert_{H^1} + \Vert \omega_0\Vert_{X_{s,\tau_0}}) \frac{1+\aalpha}{\nu\alpha} + C \tau_0 (1+\tau_0) \Vert \omega_0 \Vert_{H^1}^2 \frac{(1+\aalpha)^2}{\nu^2 \alpha^2}.\label{eq:C0def*}
\end{align}
The proof of the theorem is hence complete, modulo the proof of estimate \eqref{eq:sobolev:growth:bound}, which is given in the Appendix (cf.~Section~\ref{app:sobolev:growth:bound}).
\end{proof}

\section{Applications to the damped Euler equations}\label{sec:damped-euler}\setcounter{equation}{0}

The initial value problem for the {\it damped}
Euler equations in terms of the vorticity $\omega = \curl u$ is
\begin{align}
&\partial_t \omega + \nu \omega + (u\cdot \nabla) \omega =  (\omega \cdot \nabla) u \llabel{eq:E1}\\
&u = K_d*\omega \llabel{eq:E2}\\
&\omega(0) = \omega_0 = \curl u_0\llabel{eq:E3},
\end{align}
where $K_d$ is the ${\mathbb T}^d$-periodic Biot-Savart kernel, and $\nu \geq 0 $ is a fixed positive parameter. Here  $u$ and $\omega$ are ${\mathbb T}^d$-periodic
functions with $\int_{{\mathbb T}^d} u = 0$, and $d=2,3$. When $d=2$ the vorticity is a scalar and the term on the right of \eqref{eq:E1} is absent. It is a classical result that if $d=2$, and for any $\nu \geq 0$, the initial value problem \eqref{eq:E1}--\eqref{eq:E3} has a global in time smooth solution in the Sobolev space $H^r$, with $r>2$. We refer the reader to \cite{Chemin,MB} for details. Moreover, in the case $d=3$, and $\nu >0$, if the initial data satisfies $\Vert \nabla u_0 \Vert_{L^\infty} < \nu/\kappa$ for some sufficiently large positive dimensional constant $\kappa$, if follows from standard energy estimates that \eqref{eq:E1}--\eqref{eq:E3} has a global in time smooth solution in $H^r$, with $r>5/2$.

For results concerning the analyticity and Gevrey-class regularity of \eqref{eq:E1}--\eqref{eq:E3}, with $\nu =0$, i.e. the classical incompressible Euler equations, we refer the reader to \cite{AM,BBe,BBeZ,Be,KV,LO}. Note that in this case one can construct explicit solutions (cf.~\cite{BTi,DM}) to \eqref{eq:E1}--\eqref{eq:E3} whose radius of analyticity is decaying for all time and hence vanishes as $t\rightarrow \infty$, both for $d=2$ and $d=3$. In this section we show that if $\nu > 0$, and either $d=2$, or if $d=3$ and the initial data is small compared to $\nu$, then this is not possible: there exists a positive constant such that the radius of analyticity of the solution never drops below it. The following is our main result.

\begin{theorem}\label{thm4}
Assume that $\nu>0$, and that the divergence-free $\omega_0$ is of Gevrey-class $s$, for some $s\geq 1$. If additionally, one of the following conditions is satisfied,
\begin{enumerate}
  \item $d=2$
  \item $d=3$ and $\Vert \nabla u_0 \Vert_{L^\infty} \leq \nu/\kappa$, for some sufficiently large positive constant $\kappa$,
\end{enumerate}
then there exists a unique global in time Gevrey-class $s$ solution to \eqref{eq:E1}--\eqref{eq:E3}, with $\omega(t) \in {\mathcal D}(\LL^r e^{\tau(t) \LL^{1/s}})$ for all $t\geq 0$, and moreover we have the lower bound
\begin{align}
  \tau(t) \geq \tau(0) e^{-\bar{C} \int_{0}^{t} e^{-\nu s/2} ds} \geq \tau(0) e^{-2 \bar{C} /\nu},
\end{align}where $\bar{C}>0$ is a constant depending only on $\omega_0$.
\end{theorem}
\begin{proof}[Proof of Theorem~\ref{thm4}]
Let us first treat the case when $d=2$, with $\nu>0$ fixed. Since $\ddiv u =0 $, it classically follows from \eqref{eq:E1} that for all $1\leq p \leq \infty$ we have
\begin{align}\llabel{eq:2d:vorticity:Lp}
\Vert \omega(t) \Vert_{L^p} \leq \Vert \omega_0 \Vert_{L^p} e^{-\nu t},
\end{align}
$t\geq 0$, and for any $r>0$ the Sobolev energy inequality holds
\begin{align}\llabel{eq:2d:vorticity:sobolev}
\frac 12 \frac{d}{dt} \Vert \omega(t) \Vert_{H^r}^{2} + \nu \Vert \omega(t) \Vert_{H^r}^{2} \leq C\Vert \nabla u(t) \Vert_{L^\infty} \Vert \omega(t) \Vert_{H^r}^{2},
\end{align} where $C$ is a positive dimensional constant depending on $r$. Moreover, if $r>1$ the classical potential estimate(cf.~\cite{BrGa,MB})
\begin{align*}
\Vert \nabla u \Vert_{L^\infty} \leq C \Vert \omega \Vert_{L^2} + C \Vert \omega \Vert_{L^\infty} + C \Vert \omega \Vert_{L^\infty}\log \left(1 + \frac{\Vert \omega \Vert_{H^r}}{\Vert \omega \Vert_{L^\infty}}\right)
\end{align*}combined with \eqref{eq:2d:vorticity:Lp} shows that
\begin{align}\llabel{eq:2d:velocity:gradient}
\Vert \nabla u(t) \Vert_{L^\infty} &\leq  C e^{-\nu t} \left( \Vert \omega_0 \Vert_{L^2} + \Vert \omega_0 \Vert_{L^\infty}  + \Vert \omega_0 \Vert_{L^\infty}   \log \left(1 + \frac{e^{\nu t} \Vert \omega(t) \Vert_{H^r} }{\Vert \omega_0 \Vert_{L^\infty}}\right) \right) \notag\\
& \leq C C_0 e^{-\nu t} \left( 2 + \log \left(1 + e^{\nu t} \Vert \omega(t) \Vert_{H^r} /C_0\right) \right),
\end{align}where $C_0 = \max\{\Vert \omega_0 \Vert_{L^2},\Vert \omega_0 \Vert_{L^\infty}\} >0 $. Multiplying \eqref{eq:2d:vorticity:sobolev} by $e^{\nu t}$ and combining with the above estimate \eqref{eq:2d:velocity:gradient}, upon letting $y(t) = e^{\nu t} \Vert \omega(t) \Vert_{H^r}  / C_0$, we obtain
\begin{align*}
  \dot{y}(t)\leq C e^{-\nu t} y(t) \left( 2 + \log (1 + y(t))\right).
\end{align*}By Gr\"onwall's inequality, the above implies that there exists a positive constant $C_1 = C(C_0,\nu,\Vert \omega_0 \Vert_{H^r})$ such that $y(t) \leq C_1/C_0$ for all $t\geq 0$, and therefore by the definition of $y(t)$ we have
\begin{align}
  \Vert \omega(t) \Vert_{H^r} \leq C_1 e^{- \nu t}, \llabel{eq:2d:sobolev}
\end{align} for all $t\geq 0$. Similarly, by \eqref{eq:2d:velocity:gradient}, there exists $C_2 = C(C_0,C_1) >0$ such that for all $t\geq 0$ we have
\begin{align}
  \Vert \nabla u(t) \Vert_{L^\infty} \leq C_2 e^{- \nu t} \llabel{eq:2d:velocity}.
\end{align}
We now turn to the corresponding Gevrey-class estimates. For $r>5/2$, and initial vorticity satisfying $\Vert \LL^{r+1/2s} e^{\tau_0 \LL^{1/s}} \omega_0 \Vert_{L^2} < \infty$, the following estimate can be deduced from \cite{LO}
\begin{align}
  \frac 12 \frac{d}{dt} \Vert \LL^r e^{\tau \LL^{1/s}}\omega \Vert_{L^2}^2 + \nu \Vert \LL^r e^{\tau \LL^{1/s}}\omega \Vert_{L^2}^2 &\leq C \Vert \omega\Vert_{H^r}^3  + \left(\dot{\tau}+ C \tau \Vert \LL^r e^{\tau \LL^{1/s}}\omega \Vert_{L^2} \right) \Vert \LL^{r+1/2s} e^{\tau \LL^{1/s}}\omega \Vert_{L^2}^2. \llabel{eq:LO}
\end{align}Therefore, if $\tau(t)$ decays fast enough so that $\dot{\tau}(t) + C \tau(t) \Vert \LL^r e^{\tau(t) \LL^{1/s}} \omega(t) \Vert_{L^2} \leq 0$ for all $t\geq 0$, then using \eqref{eq:2d:sobolev} we have
\begin{align}
   \frac 12 \frac{d}{dt} \Vert \LL^r e^{\tau(t) \LL^{1/s}} \omega(t) \Vert_{L^2}^2 + \nu \Vert \LL^r e^{\tau(t) \LL^{1/s}} \omega(t) \Vert_{L^2}^2\leq C C_{1}^{3} e^{-3\nu t},
\end{align}and hence there exists a positive constant $C_3 = C( C_1, \nu, \Vert \LL^r e^{\tau_0 \LL^{1/s}}\omega_0 \Vert_{L^2})$ such that for all $t\geq 0$
\begin{align}
  \Vert \LL^r e^{\tau(t) \LL^{1/s}}\omega(t) \Vert_{L^2} \leq C_3 e^{-\nu t /2}.
\end{align} Then it is sufficient to impose
\begin{align}
  \dot{\tau}(t) + C C_3 \tau(t) e^{-\nu t /2} = 0,
\end{align}and hence we obtain the lower bound for the radius of Gevrey-class regularity
\begin{align}
  \tau(t) \geq \tau_0 e^{-C C_3 \int_{0}^{t} e^{-\nu s /2} ds}.
\end{align}
In particular it follows that for all $t\geq 0$,
\begin{align}
  \tau(t) \geq \tau_0 e^{-2 C C_3/\nu},
\end{align}which proves the first part of the theorem. The case $d=3$ is treated similarly: the estimate \eqref{eq:LO} holds also if $d=3$, so the missing ingredient is the exponential decay of the Sobolev norms. But as noted earlier, the smallness condition on $\Vert\nabla u\Vert_{L^\infty}$, not only gives the global in time existence of $H^r$ solutions, but also their exponential decay. We omit further details. \end{proof}
%\begin{remark} \llabel{rem:KV:LO}
%  Due to the damping coefficient, in the above analysis all norms $\Vert \nabla u \Vert_{L^\infty}$, $\Vert \omega \Vert_{H^r}$, and $\Vert \omega \Vert_{X_\tau}$ decay exponentially in time. Therefore the decay condition $\dot{\tau} + C \tau  \Vert \omega \Vert_{X_\tau} \leq 0$ \cite{LO} on $\tau(t)$, gives the same lower bounds as the decay condition $\dot{\tau} + C \tau \Vert \nabla u \Vert_{L^\infty} + C \tau^2 \Vert \omega \Vert_{H^r} + C \tau^2 \Vert \omega \Vert_{X_\tau} \leq 0$ \cite{KV}.
%\end{remark}

\section{Appendix}\label{sec:appendix} \setcounter{equation}{0}
\subsection{Proof of Lemma~\ref{lemma:2d:new:estimate}}\label{app:2d:new:estimate}
\begin{proof}[Proof of \eqref{eq:2d:T1}]Recall that we need to bound the quantity
\begin{align}
  T_1 &= \aalpha \left| \langle e^{\tau \LL} \left( (u\cdot \nabla) \Delta \rot u\right), e^{\tau \LL} \Delta \rot u \rangle\right|\notag\\
  &= \aalpha\left| \langle e^{\tau \LL} \left( (u\cdot \nabla) \Delta \rot u\right), e^{\tau \LL} \Delta \rot u \rangle - \langle (u\cdot \nabla) e^{\tau \LL}  \Delta \rot u, e^{\tau \LL} \Delta \rot u \rangle \right|,
\end{align}
since $\ddiv u = 0$. By Plancherel's theorem we have
\begin{align}
T_1 &\leq C \aalpha \sum\limits_{j+k=l;\, j,k,l\neq 0} \left(  e^{\tau |l|} - e^{\tau |k|}\right) |\Fu_j\cdot j| |k|^2 |k\times \Fu_k| |l|^2 |l \times \Fu_l| e^{\tau |l|}.
\end{align}
Since $|e^{\tau |l|} - e^{\tau|k|}| \leq C \tau |j| e^{\max\{|k|,|l|\}}$, we obtain
\begin{align}
T_1 &\leq C \aalpha \tau \sum\limits_{j+k=l;\, j,k,l\neq 0} |j|^2 |\Fu_j|e^{\tau |j|}  |k|^2 |k\times \Fu_k| e^{\tau |k|} |l|^2 |l \times \Fu_l| e^{\tau |l|}\notag\\
& \leq C \aalpha \tau \sum\limits_{j+k=l;\, j,k,l\neq 0;\, |l|\geq |k|} |j|^{3/2} |\Fu_j|e^{\tau |j|}  |k|^2 |k\times \Fu_k| e^{\tau |k|} |l|^{5/2} |l \times \Fu_l| e^{\tau |l|} \notag\\
& \leq C \aalpha \tau \Vert \LL^{1/2} e^{\tau \LL} \rot \Delta u \Vert_{L^2} \Vert e^{\tau \LL} \rot \Delta u \Vert_{L^2} \sum\limits_{j\neq 0} |j|^{3/2} |\Fu_j| e^{\tau |j|} \notag\\
& \leq C \aalpha \tau \Vert \LL^{1/2} e^{\tau \LL} \rot \Delta u \Vert_{L^2} \Vert e^{\tau \LL} \rot \Delta u \Vert_{L^2}^2.\llabel{eq:T1:1}
\end{align}In the above we have used the triangle inequality $|j|^{1/2} \leq |k|^{1/2} + |l|^{1/2}$, the fact that in the two-dimensional case we have $\sum_{j\in{\mathbb Z}^{2} \setminus \{0\}} |j|^{-3} < \infty$, and the Cauchy-Schwartz inequality.
%{\bf Remark:} The above bound \eqref{eq:T1:1} is independent of the dimension. The dimension-dependent optimal bounds that can be obtained on $T_1$ are
%\begin{itemize}
%  \item $d=2$
%  \begin{align}\llabel{eq:T1:2d}
%  T_1 \leq C \alpha \tau \Vert \LL^{1/2} e^{\tau \LL} \rot \Delta u \Vert_{L^2} \Vert e^{\tau \LL} \rot \Delta u \Vert_{L^2}^{2-\theta} \Vert \LL e^{\tau \LL} \rot u \Vert_{L^2}^{\theta},
%  \end{align}
%  for any $\theta\in[0,1/2)$. Note that \eqref{eq:T1} follows from \eqref{eq:T1:2d} with $\theta=0$, and the Cauchy-Schwartz inequality.
%  \item $d=3$
%    \begin{align}\llabel{eq:T1:3d}
%  T_1 \leq C \alpha \tau \Vert \LL^{1/2} e^{\tau \LL} \rot \Delta u \Vert_{L^2}^{1+\theta} \Vert e^{\tau \LL} \rot \Delta u \Vert_{L^2}^{2-\theta},
%  \end{align}
%  for any $\theta \in (0,1]$.
%\end{itemize}
The proof of \eqref{eq:2d:T1} is concluded by estimating the right side of \eqref{eq:T1:1} as
\begin{align}
  \frac{\nu}{4} \Vert e^{\tau \LL} \rot \Delta u \Vert_{L^2}^2 + \frac{C \alpha^4 \tau^2}{\nu} \Vert \LL^{1/2} e^{\tau \LL} \rot \Delta u \Vert_{L^2}^2 \Vert e^{\tau \LL} \rot \Delta u \Vert_{L^2}^2
\end{align}
\end{proof}

\begin{proof}[Proof of \eqref{eq:2d:T2}]
Recall that we need to bound the quantity $T_2$, which can be written as
\begin{align}
  T_2 &= \left| \langle \LL e^{\tau \LL} \left( (u\cdot \nabla) \rot u\right), \LL e^{\tau \LL} \rot u  \rangle - \langle (u\cdot \nabla) \LL e^{\tau \LL} \rot u, \LL e^{\tau \LL} \rot u  \rangle \right|,
\end{align}using the fact that $\ddiv u = 0$. By Plancherel's theorem we have
\begin{align}
T_2 &\leq C \sum\limits_{j+k=l;\, j,k,l\neq 0} \left(|l|  e^{\tau |l|} - |k| e^{\tau |k|}\right) |\Fu_j\cdot j|  |k\times \Fu_k| |l| |l \times \Fu_l| e^{\tau |l|}.
\end{align}By the mean value theorem, we have
\begin{align*}\left| |l|e^{\tau |l|} - |k|e^{\tau |k|} \right| \leq |j| (1 + \tau \max\{|l|,|k|\}) e^{\tau \max\{|l|,|k|\}},
\end{align*}and therefore by the triangle inequality we obtain
\begin{align}
T_2 &\leq C \sum\limits_{j+k=l;\, j,k,l\neq 0}  |\Fu_j| |j|^2 e^{\tau |j|}  |k\times \Fu_k| e^{\tau |k|} |l| |l \times \Fu_l| e^{\tau |l|} \notag\\
& + C \tau \sum\limits_{j+k=l;\, j,k,l\neq 0}  |\Fu_j| |j|^2 e^{\tau |j|}  (|j|+|k|)|k\times \Fu_k| e^{\tau |k|} |l| |l \times \Fu_l| e^{\tau |l|}.
\end{align}By symmetry, and the inequality $e^x \leq 1 + x e^x$ for all $x\geq 0$, we get
\begin{align}
T_2 &\leq C \sum\limits_{j+k=l;\, j,k,l\neq 0;\, |j|\leq |l|}  |\Fu_j| |j| e^{\tau |j|}  |k\times \Fu_k| |l|^2 |l \times \Fu_l| e^{\tau |l|} \notag\\
& + C \tau \sum\limits_{j+k=l;\, j,k,l\neq 0;\, |j|\leq |k|,|l|}  |\Fu_j| |j|^{1/2} e^{\tau |j|}  |k|^{3/2} |k\times \Fu_k| e^{\tau |k|} |l|^{2} |l \times \Fu_l| e^{\tau |l|} \notag,
\end{align}and by the Cauchy-Schwartz inequality,
\begin{align}
T_2 &\leq C \Vert \rot u \Vert_{L^2} \Vert e^{\tau \LL} \rot \Delta  u \Vert_{L^2} \sum\limits_{j\neq 0} |\Fu_j| |j| e^{\tau |j|} \notag\\
    & \qquad + C \tau \Vert \LL^{3/2} e^{\tau \LL} \rot u \Vert_{L^2} \Vert e^{\tau \LL} \rot \Delta  u \Vert_{L^2} \sum\limits_{j\neq 0}  |\Fu_j| |j|^{1/2} e^{\tau |j|}.
\end{align}Note that in the two-dimensional case, by the Cauchy-Schwartz inequality we have
\begin{align}
  \sum\limits_{j\neq 0} |\Fu_j| |j| e^{\tau |j|} &= \sum\limits_{j\neq 0} \left( |j| |\Fu_j|^{1/2} e^{\tau |j|/2} \right) \left(|j|^{3/2}|\Fu_j|^{1/2} e^{\tau |j|/2} \right) |j|^{-3/2} \notag\\
  & \leq C \Vert \LL e^{\tau \LL} \curl u \Vert_{L^2}^{1/2} \Vert e^{\tau \LL} \curl \Delta u \Vert_{L^2}^{1/2}.
\end{align}Similarly,
\begin{align}
  \sum\limits_{j\neq 0} |j|^{1/2}  |\Fu_j| e^{\tau |j|} \leq \sum\limits_{j\neq 0} |j|^{2} |\Fu_j| e^{\tau |j|} |j|^{-3/2} \leq C \Vert \LL e^{\tau \LL} \curl u \Vert_{L^2},
\end{align}and therefore
\begin{align}
T_2 &\leq C \Vert \rot u \Vert_{L^2} \Vert \LL e^{\tau \LL} \curl u \Vert_{L^2}^{1/2} \Vert e^{\tau \LL} \rot \Delta  u \Vert_{L^2}^{3/2} \notag\\
& + C \tau \Vert \LL^{3/2} e^{\tau \LL} \rot u \Vert_{L^2} \Vert \LL e^{\tau \LL} \curl u \Vert_{L^2} \Vert e^{\tau \LL} \rot \Delta  u \Vert_{L^2}.
\end{align}%
%{\bf Remark:} The dimension-dependent optimal bounds that can be obtained on $T_2$ are
%\begin{itemize}
%  \item $d=2$
%  \begin{align}\llabel{eq:T2:2d}
%  T_2 &\leq C \Vert \rot u \Vert_{L^2} \Vert \LL e^{\tau \LL} \rot u \Vert_{L^2}^\theta \Vert e^{\tau \LL} \rot \Delta u \Vert_{L^2}^{2-\theta}\notag \\
%   & \qquad + C \tau \Vert \LL^{3/2} e^{\tau \LL} \rot u \Vert_{L^2} \Vert e^{\tau \LL} \rot \Delta  u \Vert_{L^2}   \Vert \LL e^{\tau \LL} \rot  u \Vert_{L^2}
%  \end{align}
%  for any $\theta\in[0,1)$. Note that \eqref{eq:T2} follows from \eqref{eq:T2:2d} with $\theta=1/2$, and the Cauchy-Schwartz inequality.
%  \item $d=3$
%    \begin{align}\llabel{eq:T2:3d}
%  T_2 &\leq C \Vert \rot u \Vert_{L^2} \Vert \LL e^{\tau \LL} \rot u \Vert_{L^2}^\theta \Vert e^{\tau \LL} \rot \Delta u \Vert_{L^2}^{2-\theta}\notag \\
%  & \qquad + C \tau \Vert \LL^{3/2} e^{\tau \LL} \rot u \Vert_{L^2} \Vert e^{\tau \LL} \rot \Delta  u \Vert_{L^2}^{1+\eta}   \Vert \LL e^{\tau \LL} \rot  u \Vert_{L^2}^{1-\eta}
%  \end{align}
%  for any $\theta \in [0,1/2)$ and $\eta \in (0,1]$.
%\end{itemize}
The above estimate and Young's inequality concludes the proof of \eqref{eq:2d:T2}.\end{proof}

\subsection{Proof of Lemma~\ref{lemma:new}}\label{app:new}
 For convenience of notation we let $\sss = 1/s$, so that $\sss \in (0,1]$. Since $\ddiv u =0$, cf.~\cite{KV,LO} we have $(u \cdot \nabla e^{\tau \Lambda^\sss}  \omega, e^{\tau \Lambda^\sss} \omega) = 0$, and therefore
\begin{align*}
  T_1 =\left| (u \cdot \nabla \omega, e^{2\tau \Lambda^\sss} \omega)\right| =  \left| (u \cdot \nabla \omega, e^{2\tau \Lambda^\sss} \omega) - (u \cdot \nabla e^{\tau \Lambda^\sss}  \omega, e^{\tau \Lambda^\sss} \omega) \right|.
\end{align*}As in \cite{FT,KV,LO}, using Plancherel's theorem we write the above term as
\begin{align}\llabel{eq:T1}
T_1 = \left| (2\pi)^3 i \sum\limits_{j+k=l} (\hat{u}_j \cdot k) (\Fw_k \cdot \bar{{\Fw}}_l) e^{\tau |l|^\sss} \left(e^{\tau |l|^\sss} - e^{\tau|k|^\sss}\right) \right|,
\end{align}where the sum is taken over all $j,k,l\in {\mathbb Z}^3 \setminus \{0\}$. Using the inequality $e^x -1 \leq x e^x$ for $x\geq 0$, the mean-value theorem, and the triangle inequality $|k+j|^\sss \leq |k|^\sss + |j|^\sss$, we estimate
\begin{align*}
  \left|e^{\tau |l|^\sss} - e^{\tau|k|^\sss}\right| \leq  \tau \bigl| |l|^\sss - |k|^\sss\bigr|  e^{\tau \max\{|l|^\sss,|k|^\sss\} } \leq C  \tau \frac{|j|}{|k|^{1-\sss} + |l|^{1-\sss}} e^{\tau |j|^\sss} e^{\tau |k|^\sss},
\end{align*}for all $\sss\in(0,1]$, where $C> 0$ is a dimensional constant. By \eqref{eq:T1}, the triangle inequality, and the Cauchy-Schwartz inequality we obtain
\begin{align}
  T_1 & \leq C \tau \sum\limits_{j+k=l} |j| |\hat{u}_j| e^{\tau |j|^\sss} |\Fw_k|e^{\tau |k|^\sss} |{\Fw}_l| e^{\tau |l|^\sss} \frac{|k|}{|k|^{1-\sss} + |l|^{1-\sss}} \notag\\
  & \leq C \tau \sum\limits_{j+k=l} |j| |\hat{u}_j| e^{\tau |j|^\sss} |\Fw_k|e^{\tau |k|^\sss} |{\Fw}_l| e^{\tau |l|^\sss} |k|^{\sss/2} \left(|j|^{\sss/2} + |l|^{\sss/2}\right)\notag\\
  & \leq C \tau \Vert e^{\tau \Lambda^\sss} \omega \Vert_{L^2} \Vert \Lambda^{\sss/2} e^{\tau \Lambda^\sss}\omega \Vert_{L^2} \sum\limits_{j\neq 0} |j|^{1+\sss/2} |\hat{u}_j| e^{\tau |j|^\sss} + C \tau \Vert \Lambda^{\sss/2} e^{\tau \Lambda^\sss}\omega \Vert_{L^2}^2 \sum\limits_{j\neq 0} |j||\hat{u}_j| e^{\tau |j|^\sss}\notag\\
  & \leq C \tau \Vert e^{\tau \Lambda^\sss} \omega \Vert_{L^2} \Vert \Lambda^{\sss/2} e^{\tau \Lambda^\sss} \omega \Vert_{L^2} \Vert \Lambda^{3+\sss/2} e^{\tau \Lambda^\sss} u \Vert_{L^2}  + C \tau \Vert \Lambda^{\sss/2} e^{\tau \Lambda^\sss}\omega \Vert_{L^2}^2 \Vert \Lambda^3 e^{\tau \Lambda^\sss} u \Vert_{L^2} \llabel{eq:T1:v2}
\end{align}In the above we used the fact that $ \sum_{j\neq 0,\, j\in {\mathbb Z}^3} |j|^{-4} <\infty$. We recall that by \eqref{eq:u:def2} we have $u = {\mathcal K}_\alpha \omega$, and therefore for $\alpha >0$ we have
\begin{align*}
  \Vert \Lambda^3 u \Vert_{L^2} \leq \frac{C}{\alpha} \Vert \omega \Vert_{L^2},
\end{align*} and similarly
\begin{align}
  \Vert \Lambda^3 e^{\tau \Lambda^\sss} u \Vert_{L^2} \leq \frac{C}{\alpha} \Vert e^{\tau \Lambda^\sss} \omega \Vert_{L^2},\ \mbox{and}\ \Vert \Lambda^{3+\sss/2}e^{\tau \Lambda^\sss} u \Vert_{L^2} \leq \frac{C}{\alpha} \Vert \Lambda^{\sss/2} e^{\tau \Lambda^\sss} \omega \Vert_{L^2}. \llabel{eq:u:H3estimate}
\end{align}By combining \eqref{eq:T1:v2} and \eqref{eq:u:H3estimate} above, we obtain for all $\tau\geq 0$, and $\sss\in(0,1]$ that
\begin{align}
T_1 \leq \frac{C\tau}{\alpha}  \Vert e^{\tau \Lambda^\sss} \omega \Vert_{L^2} \Vert \Lambda^{\sss/2} e^{\tau \Lambda^\sss} \omega \Vert_{L^2}^2,  \llabel{eq:T1:final}
\end{align} for some sufficiently large dimensional constant $C$, thereby proving \eqref{eq:convection:term}, since $\sss=1/s$.

The estimate for the vorticity stretching term is similar. By the triangle inequality and the the estimate $e^x \leq 1 + x e^x$ for all $x\geq 0$, we have
\begin{align}
  T_2 &= \left|( \omega \cdot u, e^{2\tau \Lambda^\sss} \omega) \right| = \left| (2\pi)^3 i \sum\limits_{j+k=l} (\Fw_j \cdot k) (\hat{u}_k \cdot \bar{{\Fw}}_l) e^{2\tau |l|^\sss}  \right|\notag\\
  & \leq C \sum\limits_{j+k=l} |\Fw_j| e^{\tau |j|^\sss}  |k| |\hat{u}_k| e^{\tau |k|^\sss} |{\Fw}_l| e^{\tau |l|^\sss} \notag\\
  & \leq C \sum\limits_{j+k=l} |\Fw_j| e^{\tau |j|^\sss}  |k| |\hat{u}_k| |{\Fw}_l| e^{\tau |l|^\sss} +  C \tau \sum\limits_{j+k=l} |\Fw_j| e^{\tau |j|^\sss}  |k|^{1+\sss} |\hat{u}_k| e^{\tau |k|^\sss} |{\Fw}_l| e^{\tau |l|^\sss} \notag\\
  %& \leq C \Vert \omega \Vert_{X_{s,\tau}}^2 \sum\limits_{k\neq 0} |k| |\hat{u}_k| + \frac{C}{\alpha} \tau \Vert \omega \Vert_{X_{s,\tau}} \Vert \omega \Vert_{Y_{s,\tau}}^2\notag\\
  & \leq \frac{C}{\alpha} \Vert \omega \Vert_{L^2} \Vert e^{\tau \Lambda^\sss} \omega \Vert_{L^2}^2 + \frac{C\tau}{\alpha}  \Vert e^{\tau \Lambda^\sss}\omega \Vert_{L^2} \Vert \Lambda^{\sss/2} e^{\tau\Lambda^\sss} \omega \Vert_{L^2}^2 \llabel{eq:T2est}.
\end{align} In the last inequality above we also used $\Vert \Lambda^3 u\Vert_{L^2} \leq C \Vert \omega \Vert_{L^2}/\alpha$. This proves \eqref{eq:stretching:term} and hence concludes the proof of the lemma.

\subsection{Proof of Lemma~\ref{lemma:KV}}\label{app:KV}
For ease of notation we let $\sss = 1/s$, so that $\sss\in(0,1]$. Following notations in Section~\ref{sec:3d}, for any $m\in \{1,2,3\}$, we need to estimate
\begin{align}
  T_1 = (u\cdot \nabla \omega, \Lambda_{m}^{2} e^{2\tau \LL_m^\sss} \omega),
\end{align}
and
\begin{align}
  T_2 = (\omega\cdot \nabla u, \Lambda_{m}^{2} e^{2\tau \LL_m^\sss} \omega).
\end{align}
First we bound the term $T_1$. Note that since $\ddiv u= 0$, we have $(u \cdot \nabla \Lambda_m e^{\tau \LL_m^\sss} \omega,\Lambda_m e^{\tau \LL_m^\sss} \omega)=0$, and therefore by Plancherel's theorem we have (see also \cite{KV})
\begin{align}
  T_1 &= (u\cdot \nabla \omega, \Lambda_{m}^{2} e^{2\tau \LL_m^\sss} \omega) - (u \cdot \nabla \Lambda_m e^{\tau \LL_m^\sss} \omega,\Lambda_m e^{\tau \LL_m^\sss} \omega) \notag\\
  & = i (2\pi)^3 \sum\limits_{j+k=l} \left( |l_m| e^{\tau |l_m|^\sss} - |k_m| e^{\tau |l_m|^\sss}\right) (\hat{u}_j \cdot k) (\Fw_k \cdot \bar{\Fw}_l) |l_m| e^{\tau |l_m|^\sss}, \label{eq:T1*}
\end{align}
where the summation is taken over all $j,k,l\in {\mathbb Z}^3 \setminus \{0\}$. We split the Fourier symbol arising from the commutator, namely $|l_m| e^{\tau |l_m|^\sss} - |k_m| e^{\tau |l_m|^\sss}$, in four parts (cf.~\cite{KV}) by letting
\begin{align*}
  T_{11} &= i (2\pi)^3 \sum\limits_{j+k=l} \left( |l_m| - |k_m| \right) e^{\tau |k_m|^\sss} (\hat{u}_j \cdot k) (\Fw_k \cdot \bar{\Fw}_l) |l_m| e^{\tau |l_m|^\sss},\\
  T_{12} &= i (2\pi)^3 \sum\limits_{j+k=l} |l_m|e^{\tau |k_m|^\sss} \left( e^{\tau (|l_m|^\sss-|k_m|^\sss)} - 1 - \tau (|l_m|^\sss - |k_m|^\sss) \right) (\hat{u}_j \cdot k) (\Fw_k \cdot \bar{\Fw}_l) |l_m| e^{\tau |l_m|^\sss},\\
  T_{13} &= i (2\pi)^3 \sum\limits_{j+k=l} \tau |k_m|^{1-\sss/2} e^{\tau|k_m|^\sss} \left( |l_m|^\sss - |k_m|^\sss \right) (\hat{u}_j \cdot k) (\Fw_k \cdot \bar{\Fw}_l) |l_m|^{1+\sss/2} e^{\tau |l_m|^\sss},\\
  T_{14} &= i (2\pi)^3 \sum\limits_{j+k=l} \tau (|l_m| - |k_m|) e^{\tau |k_m|^\sss} \left( |l_m|^{1-\sss/2} - |k_m|^{1-\sss/2}\right) (\hat{u}_j \cdot k) (\Fw_k \cdot \bar{\Fw}_l) |l_m|^{1+\sss/2} e^{\tau |l_m|^\sss}.
\end{align*}To isolate the term $ \Vert \nabla u \Vert_{L^\infty}$ arising from $T_{11}$ and $T_{13}$, we need to use the inverse Fourier transform and hence may not directly bound these two terms in absolute value. The key idea is to use the one-dimensional identity (cf.~\cite{KV})
\begin{align}
  |j_m + k_m | - |k_m| = j_m \sgn(k_m) + 2 (j_m + k_m) \sgn(j_m) \chi_{\{\sgn(k_m + j_m) \sgn(k_m) = -1\}},\llabel{eq:decomposition}
\end{align}an notice that on the region $\{ \sgn(k_m + j_m) \sgn(k_m) = -1\}$, we have $0\leq |k_m| \leq |j_m|$. Define the operator $H_m$ as the fourier multiplier with symbol $\sgn(k_m)$, which is hence bounded on $L^2$. From \eqref{eq:T1*}, the defintion of $T_{11}$, and \eqref{eq:decomposition}, it follows that
\begin{align}
  T_{11} & = (\partial_m u \cdot \nabla H_m e^{\tau \LL_m^\sss} \omega, \Lambda_m e^{\tau \LL_m^\sss} \omega) \notag\\
  & + i (2\pi)^3 \sum\limits_{j+k=l; \{\sgn(k_m + j_m) \sgn(k_m) = -1\}} 2 (j_m+k_m) \sgn(j_m) e^{\tau |k_m|^\sss} (\hat{u}_j \cdot k) (\Fw_k \cdot \bar{\Fw}_l) |l_m| e^{\tau |l_m|^\sss}.
\end{align}The first term in the above is bounded by the H\"older inequality by $\Vert \nabla u \Vert_{L^\infty} \Vert \omega \Vert_{X_{s,\tau}}^2$. The second term is bounded in absolute value, by making use of $e^{\tau |k_m|^\sss} \leq e + \tau^2 |k_m|^{2\sss} e^{\tau |k_m|^\sss}$, and of $|k_m|\leq |j_m|$, by the quantity
\begin{align}
  C \Vert \omega \Vert_{H^1} \Vert \omega \Vert_{X_{s,\tau}}\left( \sum\limits_{j\neq 0} |j_m| |\hat{u}_j|\right) + C \tau^2 \Vert \omega \Vert_{Y_{s,\tau}}^2 \left(\sum\limits_{j\neq 0} |j_m|^{1+\sss} |\hat{u}_j|\right).
\end{align}By the Cauchy-Schwartz inequality, and the fact that $2(\sss-3) < -3$ for all $\sss\in(0,1]$, we have
\begin{align}
\sum_{j\neq0} |j_m|^{1+\sss} |\hat{u}_j| = \sum_{j\neq 0} |j_m|^{1+\sss} |j|^{3-\sss} |\hat{u}_j| |j|^{-3+\sss} \leq C \Vert \Lambda_{m}^{1+\sss} \LL^{3-\sss} u \Vert_{L^2} \leq C \Vert \omega \Vert_{H^1}/\alpha,
\end{align} and similarly $\sum_{j\neq 0}|j_m| |\hat{u}_j| \leq C \Vert \omega \Vert_{H^1}/\alpha$. Therefore
\begin{align}
  | T_{11}| \leq C \Vert \nabla u \Vert_{L^\infty} \Vert \omega \Vert_{X_{s,\tau}}^2 + \frac{C}{\alpha} \Vert \omega \Vert_{H^1}^2 \Vert \omega \Vert_{X_{s,\tau}} + \frac{C}{\alpha} \tau^2 \Vert \omega \Vert_{H^1} \Vert \omega \Vert_{Y_{s,\tau}}^2.\llabel{eq:T11}
\end{align}To bound $T_{13}$ one proceeds exactly the same if $s=\sss=1$. If $\sss\in(0,1)$, \eqref{eq:decomposition} may not be applied directly to $|l_m|^\sss - |k_m|^\sss$. In this case, by the mean value theorem, for any $|l_m|,|k_m|\geq 0$, there exists $\theta_{m,k,l} \in (0,1)$ such that
\begin{align}
  |l_m|^\sss - |k_m|^\sss &= \sss (|l_m| -|k_m|) |k_m|^{\sss-1}\notag\\
   & \qquad + \sss (|l_m|-|k_m|) \Big( (\theta_{m,k,l} |k_m| + (1-\theta_{m,k,l}) |l_m|)^{\sss-1} - |k_m|^{\sss-1}\Big).
\end{align}It is possible to apply \eqref{eq:decomposition} to the first term in the above identity, while the second term is bounded in absolute value by $\sss(1-\sss) |j_m|^2 |k_m|^{\sss-1} / \min\{ |k_m|,|l_m|\}$. The rest of the $T_{13}$ estimate is the same as the one for $T_{11}$ and one similarly obtains
\begin{align}
  | T_{13}| \leq C \Vert \nabla u \Vert_{L^\infty} \Vert \omega \Vert_{X_{s,\tau}}^2 + \frac{C}{\alpha} \Vert \omega \Vert_{H^1}^2 \Vert \omega \Vert_{X_{s,\tau}} + \frac{C}{\alpha} \tau^2 \Vert \omega \Vert_{H^1} \Vert \omega \Vert_{Y_{s,\tau}}^2. \llabel{eq:T13}
\end{align}The term $T_{12}$ is estimated in absolute value, by making use of the inequality $|e^x - 1 - x| \leq x^2 e^{|x|}$, and of $| |l_m|^\sss - |k_m|^\sss| \leq C |j_m|/ (|k_m|^{1-\sss} + |l_m|^{1-\sss})$. It follows from the Cauchy-Schwartz inequality applied in the Fourier variables that
\begin{align}
  |T_{12}| \leq \frac{C}{\alpha} \tau^2 \Vert \omega \Vert_{X_{s,\tau}} \Vert \omega \Vert_{Y_{s,\tau}}^2.\llabel{eq:T12}
\end{align}Similarly, by using that $e^x - 1 \leq x e^x$ for all $x\geq0$, it follows that
\begin{align}
  |T_{14}| \leq \frac{C}{\alpha} \tau \Vert \omega \Vert_{H^1}^2 \Vert \omega \Vert_{X_{s,\tau}} + \frac{C}{\alpha} \tau^2 \Vert \omega \Vert_{H^1} \Vert \omega \Vert_{Y_{s,\tau}}^2.\llabel{eq:T14}
\end{align}Combining the estimates \eqref{eq:T11}, \eqref{eq:T12}, \eqref{eq:T13}, and \eqref{eq:T14}, and using that $\tau(t) \leq \tau(0) \leq C$, we obtain the desired estimate on $T_1$. To estimate $T_2$, we proceed similarly. Here we do not have a commutator, and all terms are estimated in absolute value in Fourier space. We omit details and refer the interested reader to \cite[Proof of Lemma 2.5]{KV}.

\subsection{Proof of Estimate \eqref{eq:sobolev:growth:bound}} \llabel{app:sobolev:growth:bound}
If we take the inner product of \eqref{eq:w1} with $\omega$, and then with $\Delta \omega$, using the fact that $\int u\nabla\omega\Delta\omega=-\int\partial_k u_i\, \partial_i \omega_j\, \partial_k\omega_j$ by integrating by parts,  we obtain
\begin{align}
\frac{d}{2dt} \Vert \omega \Vert_{H^1}^2 + \frac{\nu}{1+\aalpha} \Vert \omega \Vert_{H^1}^2 \leq C \Vert \nabla u \Vert_{L^\infty} \Vert \omega \Vert_{H^1}^2 + |\langle \partial_k(\omega \cdot \nabla u) ,\partial_k \omega \rangle|.
\end{align}
The proof of \eqref{eq:sobolev:growth:bound} follows from the above estimate by using H\"older's inequality and Gr\"onwall's inequality and assuming that we have
\begin{align}
\Vert \omega \cdot \nabla u  \Vert_{H^1}  \leq C\|\nabla u\|_{L^\infty} \|\omega\|_{H^1}.
\end{align}The latter can be proved by using the Bony's para-differential calculus \cite{Chemin}. This inequality is equivalent to proving that
$$\|\Delta_q(\omega\cdot \nabla u)\|_{L^2}\leq C2^{-q}a_q\|\nabla u\|_{L^\infty}\|\omega\|_{H^1},$$
for some $0\leq a_q\in \ell^2({\mathbb N})$ with $\sum a_q^2\leq 1$.
 Let $\Delta_q(ab)=\Delta_q T_a b+\Delta_q T_b a+\Delta_qR(a,b)$, where
 $$\Delta_q R(a ,b)=\sum_{q'>q-3} \Delta_q(\Delta_{q'} a\tilde\Delta_{q'}b),$$
 and
 $$\Delta_q T_a b=\sum_{|q-q'|\leq 4} \Delta_q(S_{q'-1}b \Delta_{q'}a).$$
  We have  $\Delta_q(\omega \nabla u)= \Delta_q T_{\omega} \nabla  u+\Delta_q T_{\nabla  u} \omega +   \Delta_q R(\nabla  u, \omega)$. Using a Bernstein type inequality we have
 $$\|S_{q'-1}\omega\|_{L^\infty}\leq C 2^{2q'}\|\nabla u\|_{L^\infty}$$
 and also
 $$\|\Delta_{q'} \nabla u\|_{L^2}\leq C 2^{-2q'}\sup\limits_{|\alpha|=2}\|\Delta_{q'} \partial^\alpha \nabla u\|_{L^2}\leq C\alpha^{-1}2^{-2q}\|\Delta_{q'}\omega\|_{L^2}.$$
 So, we obtain
 $$\|\Delta_q T_{\omega} \nabla u\|_{L^2}\leq C\|\nabla u\|_{L^\infty}\|\Delta_{q'}\omega\|_{L^2}\leq C2^{-q} a_{q}\|\nabla u\|_{L^\infty}\|\omega\|_{H^1},$$
 where $a_{q}\in\ell^2({\mathbb N})$.
Similarly, we have
$$\|\Delta_q T_{\nabla u}\omega\|_{L^2}\leq C\|\nabla u\|_{L^\infty}\|\Delta_{q'}\omega\|_{L^2}\leq C2^{-q}a_q\|\nabla u\|_{L^\infty}\|\omega\|_{H^1}.$$
Concerning the rest term, we have
\begin{align}\|\Delta_q R({\omega},\nabla u)\|_{L^2}&\leq \sum_{q'>q-3}\|\Delta_{q'}\omega\|_{L^\infty}\|\tilde\Delta_{q'} \nabla u\|_{L^2}\notag\\
&\leq \sum_{q'>q-3}\|\nabla u\|_{L^\infty}\|\Delta_{q'} \omega\|_{L^2}\notag\\
&\leq C\sum_{q'>q-3}2^{-q'}a_{q'}\|\nabla u\|_{L^\infty}\|\omega\|_{H^1}\leq C2^{-q}\tilde a_q\|\nabla
u\|_{L^\infty}\|\omega\|_{H^1}
\end{align}
where $\tilde a_q=\sum_{q'>q-3}2^{-(q'-q)} a_{q'}\in \ell^2({\mathbb N})$. This complete the proof.

\end{document}